\documentclass[review]{elsarticle}

\usepackage{hyperref}

 \usepackage{paralist}
\usepackage{amsmath}
\usepackage{amsthm} 
\usepackage{stmaryrd}
\usepackage{epsfig}
\usepackage{graphics}
\usepackage{amssymb}
\usepackage{verbatim,enumerate}

\numberwithin{equation}{section}
\numberwithin{figure}{section}

\theoremstyle{plain}
\newtheorem{theorem}{Theorem}[section]
\newtheorem{lemma}{Lemma}[section]

\theoremstyle{definition}

\newtheorem{remark}{Remark}[section]

\def\be{\begin{equation}\displaystyle}
\def\ee{\end{equation}}
\def\bel{\begin{equation} \displaystyle \begin{array}{l} }
\def\eel{\end{array} \end{equation} }
\def\bell{\begin{equation} \displaystyle \begin{array}{ll}  }
\def\eell{\end{array} \end{equation} }

\def\bea{\begin{eqnarray}}
\def\eea{\end{eqnarray} }
\def\beas{\begin{eqnarray*}}
\def\eeas{\end{eqnarray*} }

\def\RR{\mathbb{R}}

\def\bx{\mathbf{x}}

\def\bar#1{{\overline #1}}
\def\R2+{\RR ^2_+}

\newcommand{\Og}{\Omega}

\newcommand{\gm}{\gamma}

\newcommand{\vep}{\varepsilon}

\newcommand{\ba}{\begin{array}}
\newcommand{\ea}{\end{array}}

\newcommand{\tphi}{\tilde{\phi}}
\DeclareMathOperator*{\argmin}{arg\,min}

\journal{Physica D}

\begin{document}

\begin{frontmatter}

\title{Ground states and energy
asymptotics of the nonlinear Schr\"{o}dinger equation}



\author[]{Xinran Ruan\corref{mycorrespondingauthor}}
\ead{a0103426@u.nus.edu}

\address{Department of Mathematics, National University of Singapore,
Singapore, 119076}

\begin{abstract}
We study analytically the existence and uniqueness of the ground state of the nonlinear Schr\"{o}dinger equation (NLSE) with a general power nonlinearity described by the power index $\sigma\ge0$.
For the NLSE under a box or a harmonic potential, we can derive explicitly the 
approximations of the ground states and their corresponding energy and chemical potential 
in weak or strong interaction regimes with a fixed nonlinearity $\sigma$.
Besides, we study the case where the nonlinearity $\sigma\to\infty$ with a fixed interaction strength.
In particular, a bifurcation in the ground states is observed. 
Numerical results in 1D and 2D will be reported to support our asymptotic results.
\end{abstract}

\begin{keyword}
nonlinear Schr\"{o}dinger equation\sep
 ground state\sep energy asymptotics\sep repulsive interaction
\MSC[2010] 35B40\sep 35P30\sep 35Q55\sep 65N25
\end{keyword}

\end{frontmatter}


\section{Introduction}

In this paper, we will consider the dimensionless
time-independent nonlinear Schr\"{o}dinger equation (NLSE)
in $d$ dimensions ($d=3,2,1$) \cite{Bao,Bao2013,BJP,Dalfovo1,Pitaevskii,Sulem}
\begin{equation}\label{eq:eig}
\left[-\frac{1}{2}\Delta+V(\mathbf{x})+\beta|\phi(\bx)|^{2\sigma}\right]\phi(\bx)=
\mu\phi(\mathbf{x}),\qquad \mathbf{x}\in\Omega\subseteq \mathbb{R}^d,
\end{equation}
where $\phi:=\phi(\bx)$ is the wave function (or eigenfunction) satisfying the normalization condition
\be\label{norm}
\|\phi\|_2^2:=\int_{\Og} |\phi(\bx)|^2d\bx=1,
\ee
$V:=V(\bx)$ is a given real-valued potential bounded below,
$\beta\ge0$ is a dimensionless constant describing
the repulsive (defocussing) interaction strength,
$\sigma\ge0$ represents different nonlinearities,
and the eigenvalue (or chemical potential in physics literature)
$\mu:=\mu(\phi)$ is defined as \cite{Bao,Bao2013,Dalfovo1,Pitaevskii}
\begin{equation}\label{def:mu}
\mu(\phi)=E(\phi)+\frac{\sigma\beta}{\sigma+1}\int_{\Omega}|\phi(\bx)|^{2\sigma+2}d\mathbf{x},
\end{equation}
with the energy $E:=E(\phi)$ defined as \cite{Bao2013,Sulem}
\begin{equation}\label{def:E}
E(\phi)=\int_{\Omega}\left[\frac{1}{2}|\nabla\phi(\bx)|^2+V(\mathbf{x})|\phi(\bx)|^2+
\frac{\beta}{\sigma+1}|\phi(\bx)|^{2\sigma+2}\right]d\mathbf{x}.
\end{equation}
If $\Omega$ is bounded, the homogeneous
Dirichlet BC, i.e. $\phi(\bx)|_{\partial \Omega}=0$, needs to be imposed.
Thus, the time-independent NLSE (\ref{eq:eig}) is a nonlinear eigenvalue problem under
the constraint $\|\phi\|=1$. It is a mean field model arising from  Bose-Einstein condensates
(BECs) \cite{Anderson,Bao,Dalfovo1,GPE_BEC1}, nonlinear optics \cite{optics},
and some other applications \cite{Abl,Pitaevskii,Sulem}
that can be obtained from
the N-body Schr\"{o}dinger equation via the Hartree ansatz and mean field approximation \cite{Bao2013,MFT,TGmath,Pitaevskii}.
When $\beta=0$ or $\sigma=0$, it collapses to the time-independent Schr\"{o}dinger equation. 
When $\sigma=1$, the nonlinearity is cubic and it is usually known as the Gross-Pitaevskii equation (GPE) \cite{Bao,Dalfovo1,Dalfovo2,Pitaevskii}.
When $\sigma=2$, the nonlinearity is quintic and it is used to model the Tonks-Girardeau (TG) gas in BEC \cite{Girardeau1,Lieb_model,TGmath,Tonks}.

The ground state of the NLSE (\ref{eq:eig}) is usually defined as the minimizer
of the non-convex minimization problem (or constrained minimization problem) \cite{Bao,Bao2013,Dalfovo1,GPE_BEC1}
\begin{equation}\label{def:ground}
\phi_g=\argmin_{\phi\in{S}}E(\phi),
\end{equation}
where $S=\{\phi\, |\, \|\phi\|_2^2:=\int_{\Omega} |\phi(\bx)|^2d\bx =1,\ E(\phi)<\infty,
\ \phi|_{\partial\Omega}=0 \ \hbox{if} \ \Omega
\ \hbox{is bounded}\}$. 
Since $S$ is a nonconvex set,  the problem (\ref{def:ground}) is a nonconvex minimization
problem.
It is easy to see that the ground state $\phi_g$ satisfies the
time-independent NLSE (\ref{eq:eig}). Hence
it is an eigenfunction (or stationary state) of (\ref{eq:eig}) with the least energy.

The main purpose of this paper is to study the existence and uniqueness of the ground state of the NLSE and its approximations under a box or a harmonic potential in special parameter regimes.
The rest of this paper is organized as follows. 
In Section \ref{existence}, we study analytically the existence, uniqueness and nonexistence of the ground state of the NLSE.
In Section \ref{harb:NLSE}, we derive the ground state approximations and energy asymptotics under a harmonic potential for different $\beta$'s and $\sigma$'s.
Similar results are presented in Section \ref{box:NLSE} for the NLSE under a box potential.
Some conclusions are drawn in Section \ref{conclusion}.

\section{Existence and uniqueness}\label{existence}
In this section, 
we will generalize the existence and uniqueness results for the GPE case \cite{Bao2013,Lieb,Wein}, where $\sigma=1$, to a general case with a nonnegative $\sigma$. 
For simplicity, we introduce the function space
\begin{equation*}
X=\left\{\phi\in H^1(\mathbb R^d)\left| \|\phi\|_X^2=
\|\phi\|^2+\|\nabla\phi\|^2+\|\phi\|_{L_V}<\infty\right.\right\}.
\end{equation*}
where $\|\phi\|_{L_V}:=\int_{\mathbb R^d}V(\bx)|\phi(\bx)|^2\,d\bx$. 
The following embedding results hold \cite{Bao2013}.
\begin{lemma}\label{lem:compact} Under the assumption that $V(\bx)$ is nonnegative and satisfies the  confining condition,
i.e. $\lim\limits_{R\to\infty} V(\bx)=\infty$, for $\bx\in {\mathbb R}^d$ where $d=1,2,3$, 
we have that
the embedding $X\hookrightarrow L^{p}(\mathbb R^d)$
is compact provided that exponent $p$
 satisfies
\begin{equation}
\begin{cases}
p\in[2,6),\quad d=3,\\
p\in[2,\infty),\quad d=2,\\
p\in [2,\infty],\quad d=1.
\end{cases}
\end{equation}
\end{lemma}
In the $d$-dimensional space, where $d=1,2,3$, let $C_b(d,\sigma)$ be the
best constant in the following inequality \cite{Wein}
\be\label{bestc}
C_b(d,\sigma):=\inf_{0\ne f\in H^1({\mathbb R}^d)} \frac{\|\nabla
f\|^{d\sigma}\|f\|^{2+(2-d)\sigma}}{\|f\|_{{2\sigma+2}}^{2\sigma+2}}.
\ee

Then we have the following theorem regarding the existence and uniqueness of the ground state.
\begin{theorem}\label{thm:mres}(Existence and uniqueness)
Suppose  $V(\bx)\ge 0$ satisfies the confining condition, i.e.
$\lim\limits_{|\bx|\to\infty}V(\bx)=+\infty$, where $\bx\in\mathbb{R}^d$, then  there exists a
minimizer
 $\phi_g\in S$ for
 \eqref{def:ground} if one of the following conditions holds

(i)  $\beta\in\mathbb{R}$ for $0<d\sigma<2$;

(ii) $\beta>-\frac{(\sigma+1)}{2}C_b(d,\sigma)$ when $d\sigma=2$;

(iii) $\beta\ge0$ for $d\sigma>2$.

Furthermore, the ground state  can be chosen as nonnegative $|\phi_g|$, and $\phi_g=e^{i\theta}|\phi_g|$ for some constant $\theta\in\mathbb{R}$.
For $\sigma>0$ and $\beta\ge0$, the nonnegative ground state is unique. 

In contrast, there exists no ground state if one of the following conditions holds

(i') $\beta<-\frac{(\sigma+1)}{2}C_b(d,\sigma)$ when $d\sigma=2$;

(ii') $\beta<0$ for $d\sigma>2$.

\end{theorem}

\begin{proof}
  We separate the proof into the existence and nonexistence part.

(1) Existence. The inequality \cite{LiebL}
\be
|\nabla|\phi(\bx)||\leq |\nabla\phi(\bx)|,\quad \text{a.e.}\quad \bx\in\Bbb R^d,
\ee
implies
\be
E(\phi)\ge E(|\phi|),
\ee
where the equality holds iff $\phi=e^{i\theta}|\phi|$ for some constant $\theta\in\mathbb{R}$. Therefore, it suffices to consider the real-valued functions   
for the rest part of the proof.

We first claim that the energy is bounded from below under the condition (i), (ii) or (iii).
For case (iii), it is trivial to see  that the energy is bounded below by 0. 
For case (i) or (ii), the lower boundedness can be shown via the Gagliardo-Nirenberg inequality.
For any $\phi\in S$, 
the Gagliardo-Nirenberg inequality implies 
\be\label{G-N}
\|\phi\|_{L^{2\sigma+2}(\mathbb{R}^d)}\le \frac{1}{C_b(d,\sigma)}\|\nabla \phi\|^{\frac{d\sigma}{2\sigma+2}}\|\phi\|^{\frac{2+(2-d)\sigma}{2\sigma+2}},
\ee
where $\sigma$ is required to satisfy $0<\sigma\le2$ when $d=3$ and $\sigma>0$ when $d=1$ or 2. 
For case (i), noticing $d\sigma<2$ and applying the H\"{o}lder's inequality, we have 
\be
\|\phi\|_{L^{2\sigma+2}(\mathbb{R}^d)}^{2\sigma+2}\le \frac{1}{C_b(d,\sigma)}\|\nabla \phi\|^{d\sigma}\|\phi\|^{2+(2-d)\sigma}\le \vep\|\nabla \phi\|^{2}+C(d,\sigma,\vep),
\ee 
which yields the claim by choosing a sufficiently small $\vep$.\\
For case (ii), we have $d\sigma=2$ and the inequality \eqref{G-N} now becomes
\be
\|\phi\|_{L^{2\sigma+2}(\mathbb{R}^d)}^{2\sigma+2}\le \frac{1}{C_b(d,\sigma)}\|\nabla \phi\|^2\|\phi\|^{4/d}=\frac{1}{C_b(d,\sigma)}\|\nabla \phi\|^2,
\ee 
which implies that if $\beta>-\frac{(\sigma+1)}{2}C_b(d,\sigma)$, we have 
\be
\frac{1}{2}\|\nabla \phi\|^2+\frac{\beta}{\sigma+1}\|\phi\|_{L^{2\sigma+2}(\mathbb{R}^d)}^{2\sigma+2}\ge 0,
\ee
and therefore the energy \eqref{def:E} is bounded from below.

Hence for cases under condition (i), (ii) or (iii), we can take a sequence $\{\phi^n\}_{n=1}^{\infty}$ to minimize
the energy \eqref{def:E} in $S$, and the sequence is uniformly bounded in $X$. Taking a
weakly convergent subsequence, which is denoted as the original sequence for simplicity, we have
\be
\phi^n\rightharpoonup\phi^{\infty}, \text{ weakly in $X$}.
\ee
Lemma \ref{lem:compact} ensures that $\{\phi^n\}_{n=1}^{\infty}$ converges to $\phi^{\infty}$ in $L^p$ where $p$ is given in Lemma \ref{lem:compact}, and we get $\|\phi^{\infty}\|=1$ in particular by taking $p=2$.
Further,  from the lower-semicontinuity of the $H_1$ norm and Fatou's lemma, we can show 
$E(\phi^{\infty})\le\liminf_{n\to\infty} E(\phi^n)$, 
which implies that  $\phi^{\infty} \in S$ is indeed a ground state.
Thus we proved the existence of the ground state. 
When $\beta\ge0$ and $\sigma>0$, 
 the uniqueness of the ground state comes from the strict convexity of the energy functional.
 
 (2) Nonexistence. We take $\phi(\bx)=\pi^{-d/4}e^{-|\bx|^2/2}$ and $\phi^{\vep}(\bx)=\vep^{-d/2}\phi(\bx/\vep)$. 
 It is easy to check that $\|\phi^{\vep}(\bx)\|=\|\phi(\bx)\|=1$ for all $\vep>0$ and
 \be
 E(\phi^\vep)=\frac{\|\nabla\phi\|^2}{2}\frac{1}{\vep^2}+\frac{\beta\|\phi\|_{{2\sigma+2}}^{2\sigma+2}}{(\sigma+1)}\frac{1}{\vep^{d\sigma}}+O(1).
 \ee
Under condition (ii'), i.e. $\beta<0$ and $d\sigma>2$, we have 
$
E(\phi^\vep) \to-\infty
$
 as $\vep\to0^+$ and therefore there exists no ground state. For case (i'), we have $-\frac{2\beta}{\sigma+1}>C_b(d,\sigma)$ and therefore may choose $\phi_b(x)$ satisfying $\|\phi_b\|=1$ and 
 $\frac{1}{2}\|\nabla
\phi_b\|^{2}+\frac{\beta}{\sigma+1}\|\phi_b\|_{{2\sigma+2}}^{2\sigma+2}<0$.
Then we have
\be
 E(\phi_b^\vep)=\left(\frac{1}{2}\|\nabla\phi_b\|^2+\frac{\beta}{\sigma+1}\|\phi_b\|^{2\sigma+2}_{{2\sigma+2}}\right)\frac{1}{\vep^{2}}+O(1)\to-\infty \text{ as } \vep\to0^+.
\ee
 Therefore, under condition (i') or (ii'), $E(\phi^\vep)$ is not bounded from below and thus there exists no ground state.
\end{proof}

\section{Approximations under a harmonic potential}\label{harb:NLSE}
\setcounter{equation}{0}
\setcounter{figure}{0}
In this section, we take $\Omega={\mathbb R}^d$ and the external potential $V(\mathbf{x})=\sum_{j=1}^d\frac{\gm_j^2x_j^2}{2}$ to be a harmonic potential 
with $\gm_j>0$. 
We denote the ground state as $\phi_g^{\beta,\sigma}(\bx)$ for given $\beta\ge0$ and $\sigma\ge0$, and
denote the corresponding energy and chemical potential as
$E_g(\beta,\sigma)=E(\phi_g^{\beta,\sigma})$ and $\mu_g(\beta,\sigma)=\mu(\phi_g^{\beta,\sigma})$,
respectively.
When $\sigma=0$, the NLSE \eqref{eq:eig} collapses to a linear Schr\"{o}dinger
equation, which has been well studied. 
From now on, we consider the case $\sigma>0$ only.

\subsection{For different $\beta$ under a fixed $\sigma>0$}\label{gGPE_har}
For problems with fixed nonlinearity, Theorem \ref{thm:mres} indicate that the ground state exists for all $\beta$ if $d\sigma<2$ and for all $\sigma>0$ if $\beta\ge0$. 
Therefore, the limiting behavior of the ground state as $\beta\to0$ or $\beta\to\infty$ for general $\sigma>0$ and $\beta\to-\infty$ for $\sigma$ satisfying $d\sigma<2$ will be of great interest. 
For simplicity, only the isotropic case $\gm_1=\gm_2=\dots=\gm_d=\gm$ is considered in this section.

When $0< \beta\ll1$, the results can be summarized in the following lemma. 

\begin{lemma}\label{lemma:weak_nlse}
When $0< \beta\ll1$, i.e. the weakly repulsive interaction regime,
the ground state $\phi_g^{\beta,\sigma}$ can be approximated as
\begin{equation}\label{harg678}
\phi_g^{\beta,\sigma}(\mathbf{x})\approx\phi_g^{0}
(\mathbf{x}):=\prod_{j=1}^d\left(\frac{\gm}{\pi}\right)^{\frac{1}{4}}e^{-\frac{\gm{x}_j^2}{2}},
\qquad \bx\in{\mathbb R}^d,
\end{equation}
and the corresponding energy and chemical potential can be approximated as
\begin{align}\label{nlsehar1}
E_g(\beta,\sigma)&=\frac{d\gm}{2}+\frac{\beta}{(\sigma+1)^{\frac{d+2}{2}}}
\left(\frac{\gm}{\pi}\right)^{\frac{d\sigma}{2}}+o(\beta),\\
\label{nlsehar2}
\mu_g(\beta,\sigma)&=\frac{d\gm}{2}+\frac{\beta}{(\sigma+1)^{\frac{d}{2}}}
\left(\frac{\gm}{\pi}\right)^{\frac{d\sigma}{2}}+o(\beta).
\end{align}
\end{lemma}

\begin{proof}
When $\beta=0$, all eigenfunctions of \eqref{eq:eig} can be expressed via the Hermite functions, and the ground state is exactly $\phi_g^0$ in \eqref{harg678}. When $0<\beta\ll1$, we can approximate the ground state $\phi_g^{\beta,\sigma}$ by $\phi_g^0$. 
Plugging (\ref{harg678}) into (\ref{def:E}) and (\ref{def:mu}) with $V(\bx)=\frac{\gm^2|\mathbf{x}|^2}{2}$, we get \eqref{nlsehar1} and \eqref{nlsehar2}, respectively.
The detailed computation is omitted here for brevity.
\end{proof}
We can further prove rigorously the convergence of the approximate ground state for this case and the result is formulated as the following theorem.

\begin{theorem}\label{thm:har_limit} 
Consider the NLSE \eqref{eq:eig} with $\sigma>0$ and $V(\bx)=\sum_{j=1}^d \gm_j^2x_j^2/2$, where $d=1,2,3$. 
 When $\beta\to 0^+$ for general $\sigma>0$ or $\beta\to 0$ for $\sigma$ satisfying $d\sigma<2$, we have $\phi_g^{\beta,\sigma}(\bx)$ converges to $\phi_g^0(\bx)$ in $H^1$, where $\phi_g^0(\bx):=\prod_{j=1}^d\left(\frac{\gm_j}{\pi}\right)^{\frac{1}{4}}e^{-\frac{\gm_j{x}_j^2}{2}}$.
\end{theorem}
\begin{proof}
 We start the proof by showing the case $\beta\to0^+$ for general $\sigma>0$. Define
 \begin{equation}\label{def:E0}
E_0(\phi)=\int_{\Omega}\left[\frac{1}{2}|\nabla\phi(\bx)|^2+V(\mathbf{x})|\phi(\bx)|^2\right]d\mathbf{x}.
\end{equation}
  The minimizer of $E_0(\cdot)$ exists and is unique by the strict convexity of $E_0(\cdot)$ \cite{Bao2013}. For $0<\beta\ll1$, a simple computation shows that
\be\label{proof:har_0}
E_0(\phi_g^0)\le E_0(\phi_g^{\beta,\sigma})\le E(\beta,\sigma)\le E(\phi_g^0) \le E_0(\phi_g^0) +O(|\beta|).
\ee
Therefore, $\lim_{\beta\to0^+}E(\beta,\sigma)=E_0(\phi_g^0)$, and $\|\nabla\phi_g^{\beta,\sigma}\|$, $\|\phi_g^{\beta,\sigma}\|_{L_V}$ is uniformly bounded above, i.e. $\|\nabla\phi_g^{\beta,\sigma}\|\le C$ and $\|\phi_g^{\beta,\sigma}\|_{L_V}\le C$ for some constant $C$. 

The boundedness of the $H_1$ norm implies that there exists $\tphi\in H_1$ such that $\phi_g^{\beta,\sigma}\to\tphi$ weakly in $H_1$. 
We claim that $\phi_g^{\beta,\sigma}\to\tphi$ strongly in $L_2$. 
For any $\eta>0$, the confinement of the external potential implies that there exists $R$ such that when $|\bx|>R$, $V(\bx)>\frac{C}{\eta}$. As a result, 
\be
\int_{|\bx|>R}|\phi_g^{\beta,\sigma}|^2\,d\bx\le\frac{\eta}{C}\int_{|\bx|>R} V(\bx)|\phi_g^{\beta,\sigma}|^2\,d\bx\le\eta.
\ee
While in the bounded domain $\{ \bx\,|\, |\bx|\le R\}$, the Sobolev embedding theorem implies that $\phi_g^{\beta,\sigma}\to\tphi$ strongly in $L_2$. It follows that 
\be
\limsup_{\beta\to0^+}\int_{\mathbb{R}^d} |\phi_g^{\beta,\sigma}-\tphi|^2\,d\bx\le 4\eta.
\ee
Since $\eta$ is arbitrary, we proved our claim. Consequently, $\|\tphi\|=1$ and $\tphi\in S$.

Further, we claim that $\tphi=\phi_g^0$. 
 In fact, from the lower-semicontinuity of the $H_1$ norm and the Fatou's lemma,  we have  
\be
E_0(\tphi)\le\liminf_{\beta\to0^+}E(\beta,\sigma)=E_0(\phi_g^0).
\ee
Therefore, $\tphi$ is a ground state  
and we must have $\tphi=\phi_g^0$ by the uniqueness of the ground state. What's more, we have 
$E_0(\tphi)=\lim_{\beta\to0^+}E(\beta,\sigma)$. As a consequence, 
$\|\nabla\phi_g^0\|=\lim_{\beta\to0^+}\|\nabla \phi_g^{\beta,\sigma}\|$. Combining the equation and the weak convergence of $\phi_g^{\beta,\sigma}$, we will get $ \phi_g^{\beta,\sigma}\to\phi_g^0$ strongly in $H_1$.

For the case $\beta\to 0$ with $d\sigma<2$, the  Gagliardo-Nirenberg inequality implies that $\|\phi_g^{\beta,\sigma}\|_{2\sigma+2}^{2\sigma+2}\le\|\nabla\phi_g^{\beta,\sigma}\|+C_1$. Therefore
for the lower bound part of \eqref{proof:har_0}, 
\be
E(\beta,\sigma)\ge(1-O(|\beta|))E_0(\phi_g^{\beta,\sigma})-O(|\beta|)\ge(1-O(|\beta|))E_0(\phi_g^0)-O(|\beta|),
\ee
and thus we still have $\lim_{\beta\to0^+}E(\beta,\sigma)=E_0(\phi_g^0)$ and the uniform boundedness of $\|\nabla\phi_g^{\beta,\sigma}\|$ and $\|\phi_g^{\beta,\sigma}\|_{L_V}$. The remained part of the proof  is essentially the same and is omitted here for brevity.
\end{proof}

When $\beta\gg1$,  we have the following lemma about the approximation of the ground state.
\begin{lemma}\label{lemma:strong_nlse}
When $\beta\gg1$, i.e. the strongly repulsive interaction regime,
the ground state can be approximated as
\begin{equation}\label{sol:har_strong_gGPE}
\phi_g^{\beta,\sigma}(\mathbf{x})\approx\phi_g^{\rm TF}(\mathbf{x})=\left\{\ba{ll}
\left(\frac{\mu_g^{\rm TF}-\gm^2|\mathbf{x}|^2/2}{\beta}\right)^{\frac{1}{2\sigma}},
&\gm^2|\mathbf{x}|^2\le 2\mu_g^{\rm TF},\\
0, &{\rm otherwise},\\
\ea\right.
\end{equation}
and the corresponding energy and chemical potential can be approximated as
\begin{align}\label{nlsechm834}
&\mu_g(\beta,\sigma)\approx\mu_g^{\rm TF}=\left(\frac{\beta^{\frac{1}{\sigma}}\gm^d}{2^{\frac{d}{2}-1}
dC_dB(\frac{d}{2},1+\frac{1}{\sigma})}\right)^{\frac{1}{\frac{d}{2}+\frac{1}{\sigma}}},\\
\label{nlseeng834}
&E_g(\beta,\sigma)\approx E_g^{\rm TF}=\frac{2+d\sigma}{2\sigma+2+d\sigma}\mu_g^{\rm TF}, \qquad \beta\gg1.
\end{align}
where $B(a,b)$ is the standard beta function and $C_d=2$ when $d=1$, $C_d=\pi$ when $d=2$ and $C_d=4\pi/3$ when $d=3$.
\end{lemma}

\begin{proof}
Set $\phi_g^\vep(\bx)=\vep^{-d/2}\phi_g^{\beta,\sigma}(\bx/\vep)$ with $\vep=\beta^{-\frac{1}{2+d\sigma}}$ and define
\be\label{def:E_vep}
E^{\vep}(\phi)=\int_{\Bbb R^d}\left(\frac{\vep^4}{2}|\nabla\phi|^2+V(\bx)|\phi|^2+\frac{1}{\sigma+1}|\phi|^{2\sigma+2}\right)\,d\bx.
\ee
It is easy to check that $\vep\to0^+$ as $\beta\to+\infty$ and $\phi_g^\vep$ minimizes $E^\vep(\cdot)$ iff $\phi_g^{\beta,\sigma}$ minimizes $E(\beta,\sigma)$. 
From \eqref{def:E_vep}, it is natural to assume that the ground state $\phi_g^\vep$ converges to the ground state of the following energy as $\vep\to0$,
\be\label{eq:es}
E_\infty(\phi)=\int_{\Bbb R^d}\left(V(\bx)|\phi|^2+\frac{1}{\sigma+1}|\phi|^{2\sigma+2}\right)\,d\bx,
\ee
which drops the kinetic energy part. The minimizer of \eqref{eq:es} is usually called the Thomas-Fermi (TF) approximation in the literature. 
The Euler-Lagrange equation of \eqref{eq:es} implies the TF approximation satisfies the following equation 
\begin{equation}\label{eq:eigTF}
\frac{\gm^2|\bx|^2}{2}\phi^{\rm TF}_g(\bx)+\beta|\phi_g^{\rm TF}(\bx)|^{2\sigma}\phi_g^{\rm TF}(\bx)=
\mu_g^{\rm TF}\phi_g^{\rm TF}(\mathbf{x}),\qquad \mathbf{x}\in \mathbb{R}^d.
\end{equation}
Solving the above equation, we get (\ref{sol:har_strong_gGPE}).
Substituting (\ref{sol:har_strong_gGPE}) into \eqref{norm} and \eqref{def:mu},
 we get \eqref{nlsechm834} and \eqref{nlseeng834}, respectively.  
 The detailed computation is omitted here for brevity.
 \end{proof}

Now we consider the case $\beta\to-\infty$ when $d\sigma<2$.  
In this case, there will be a strong attractive interaction between particles. 
Set $\phi_g^\vep(\bx)=\vep^{d/2}\phi_g^{\beta,\sigma}(\vep\bx)$ with $\vep=|\beta|^{-\frac{1}{2-d\sigma}}$ and define
\be\label{def:E_vep-}
E_-^{\vep}(\phi)=\int_{\Bbb R^d}\left(\frac{1}{2}|\nabla\phi|^2+\vep^4V(\bx)|\phi|^2-\frac{1}{\sigma+1}|\phi|^{2\sigma+2}\right)\,d\bx.
\ee
Easy to see $\vep\to0^+$ as $\beta\to-\infty$. 
Again $\phi_g^\vep$ minimizes $E_-^\vep(\cdot)$ iff $\phi_g^{\beta,\sigma}$ minimizes $E(\beta,\sigma)$. 
From \eqref{def:E_vep-}, it is natural to assume that ground state $\phi_g^\vep$ converges to  the state $\phi_{-}$ which minimizes the following energy
\be\label{eq:es_-infty}
E_r(\phi)=\int_{\Bbb R^d}\left(\frac{1}{2}|\nabla\phi|^2-\frac{1}{\sigma+1}|\phi|^{2\sigma+2}\right)\,d\bx.
\ee
This implies that we can approximate the ground state $\phi_g^{\beta,\sigma}$ by the ground state of the following nonlinear eigenvalue problem when $\beta<0$ and $|\beta|\gg1$,
\be\label{eq:neg_infty}
\left[-\frac{1}{2}\Delta+\beta|\phi(\bx)|^{2\sigma}\right]\phi(\bx)=
\mu\phi(\mathbf{x}),\qquad \mathbf{x}\in\Omega\subseteq \mathbb{R}^d.
\ee
Notice that there is no external potential term in the equation \eqref{eq:neg_infty}, the approximation does not depend on the external potential we choose.

Here we provide accuracy tests for the asymptotic results shown in Lemma \ref{lemma:weak_nlse} and Lemma \ref{lemma:strong_nlse}. From Fig. \ref{fig:ghar_E_beta}, we can see that our approximations agree with the exact values very well in both weak and strong interaction regimes.
\begin{figure}[htbp]
\centerline{\psfig{figure=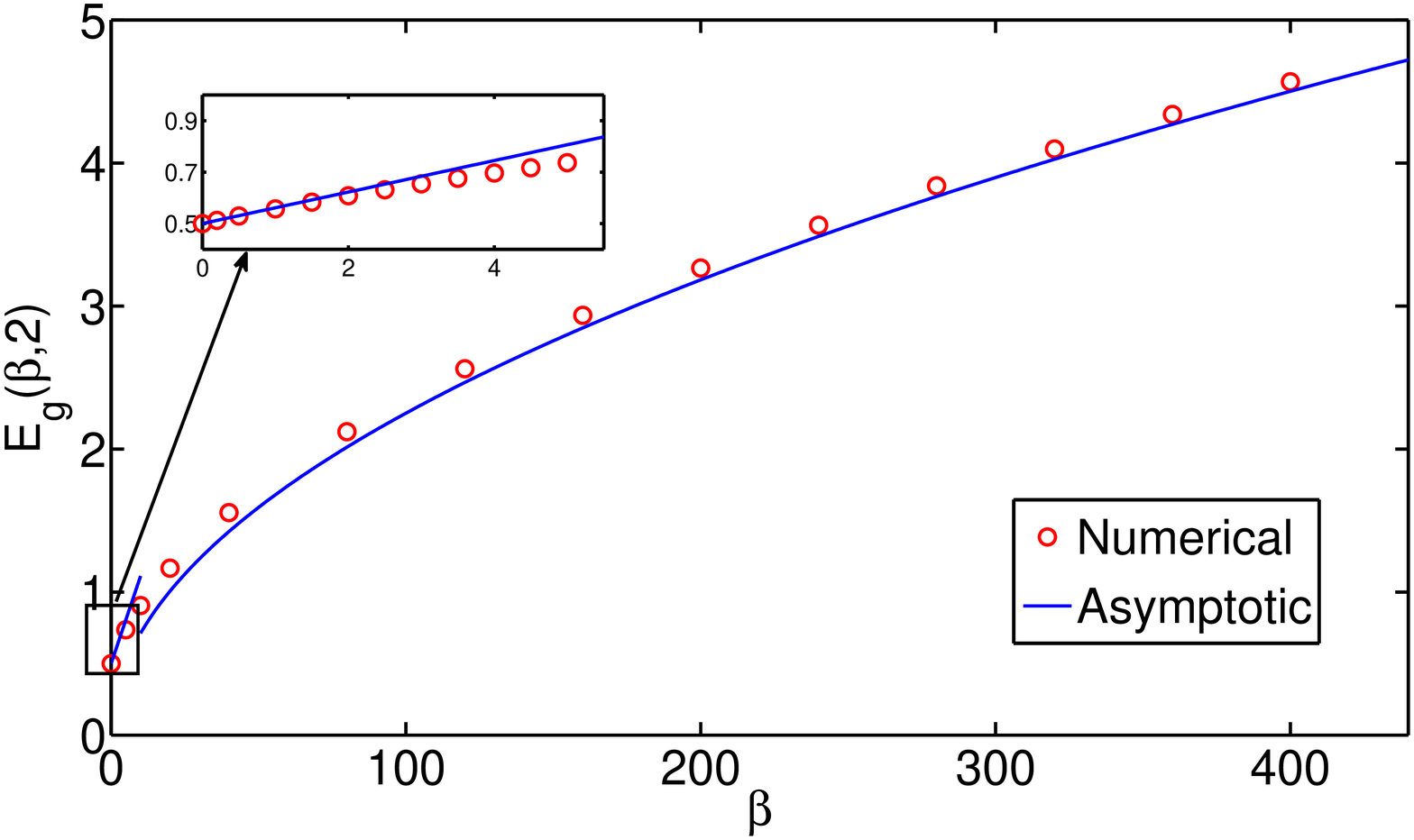,height=5cm,width=13cm,angle=0}}
\caption{The ground state energy of
the NLSE (\ref{eq:eig}) in 1D under a harmonic potential with $\sigma=2$ (quintic nonlinearity for TG gas) and $\gm=3$ for different $\beta$. }
\label{fig:ghar_E_beta}
\end{figure}

\subsection{When $\sigma\to\infty$ under a fixed $\beta>0$}
In this section, we fix $\beta>0$ and study the limit of the ground state
$\phi_g^{\beta,\sigma}$ as $\sigma\to\infty$.
For simplicity, we will only consider the NLSE \eqref{eq:eig} in
1D under the harmonic potential $V(x)=\frac{\gm^2 x^2}{2}$ for some $\gm>0$.

\begin{lemma}\label{gGPE_har_s}
For any given $\beta>0$, when $\sigma\to\infty$, we have

(i) If $0<\gm\le\pi$, the ground state converges to the linear approximation
\begin{align}\label{bnlse257}
&\phi_g^{\beta,\sigma}(x)\approx\phi_g^{0}(x)=
\left(\frac{\gm}{\pi}\right)^{\frac{1}{4}}e^{-\frac{\gm{x}^2}{2}}, \qquad x\in{\mathbb R},\\
&E_g(\beta,\sigma)\approx \frac{\gm}{2}+\frac{\beta}{(\sigma+1)^{\frac{3}{2}}}\left(\frac{\gm}{\pi}\right)^{\frac{\sigma}{2}}
\to\frac{\gm}{2},\quad
\mu_g(\beta,\sigma)\approx \frac{\gm}{2}+\frac{\beta}{(\sigma+1)^{\frac{1}{2}}}
\left(\frac{\gm}{\pi}\right)^{\frac{\sigma}{2}}\to \frac{\gm}{2}.
\end{align}

(ii) If $\gm>\pi$, the ground state converges to
\be\label{harnlse367}
\phi_g^{\beta,\sigma}(x)\to\psi_g^\gm(x)=
\begin{cases}
\varphi\left(-x\right), &x<-x_\gm,\\
1, &-x_\gm \le{x}\le x_\gm,\\
\varphi\left(x\right), &x>x_\gm,
\end{cases}
\end{equation}
where $\varphi(x)$ is the unique positive ground state of
the following linear eigenvalue problem with $\mu$ the corresponding eigenvalue
\begin{equation}\label{sol:gGPE_har_eq}
\begin{cases}
\mu\varphi(x)=-\frac{1}{2}\varphi^{\prime\prime}(x)+\frac{\gm^2x^2}{2}\varphi(s), \qquad x>x_\gm,\\
\varphi(x_\gm)=1, \qquad \varphi^{\prime}(x_\gm)=0, \qquad \lim\limits_{x\to+\infty}\varphi(x)=0,
\end{cases}
\end{equation}
with the constant $x_\gm\ge0$ determined by
\be \label{nlsehar689}
x_\gm+\int_{x_\gm}^\infty |\varphi(x)|^2dx=\frac{1}{2}.
\ee
\end{lemma}

\begin{proof}
In order to find the limit of $\phi_g^{\beta,\sigma}(x)$ when $\sigma\to\infty$,
the main idea is to determine which term on the left hand side of (\ref{eq:eig})
is negligible when $\sigma>>1$. 
Note that
\be
a^{2\sigma}\to\left\{\ba{ll}
0,  &|a|<1, \\
1, &a=1,\\
+\infty, &a>1.\\
\ea\right.
\ee
In the region where $|\phi(x)|<1$, the nonlinear term can
be dropped and we get the linear approximation, whose solution is the Gaussian function.
In the region where $|\phi(x)|>1$, the diffusion term can be dropped and we
get the TF approximation. Therefore, there are two possible cases about the
limit $\phi_g^{\beta,\sigma}(x)\to\phi^{\rm{app}}(x)$ for $x\in{\mathbb R}$ when $\sigma\to+\infty$:
(i) $|\phi^{\rm{app}}(x)|\le1$ for all $x\in{\mathbb R}$, (ii)
there exists a $x_c\ge0$ such that  $|\phi^{\rm{app}}(x)|>1$ for
$x\in[-x_c,x_c]$ and $|\phi^{\rm{app}}(x)|<1$ otherwise.

(i) When $0<\gm\le\pi$, the linear approximation suggests that
$\phi^{\rm app}(x)=\left(\frac{\gm}{\pi}\right)^{\frac{1}{4}}e^{-\frac{\gm{x}^2}{2}}\le1$ for
$x\in{\mathbb R}$. Note that the requirement that $\sup\limits_{x\in{\mathbb R}}
\phi^{\rm app}(x)\le1$ implies that $0<\gm\le\pi$.
Therefore, we get the necessary and sufficient condition about $\gm$ for
(\ref{bnlse257}) to be true.

(ii) When $\gm>\pi$, we may expect neither the linear approximation nor the
TF approximation is valid for $x\in{\mathbb R}$. Instead,
a combination of the linear approximation and TF approximation should be used.
To be more specific, for any fixed $\sigma>0$, when
$\beta>>1$, there exists a constant $x_c^\sigma\ge0$ such that when $x\in(-\infty,x_c^\sigma)\cup
(x_c^\sigma,\infty)$, the linear approximation is used; and  when $x\in[-x_c^\sigma,x_c^\sigma]$,
the TF approximation $\phi(x)=\left(\frac{\mu_g-\gm^2x^2/2}{\beta}\right)^{\frac{1}{2\sigma}}$ which
goes to the constant 1 as $\sigma\rightarrow\infty$. Therefore, we can simply use the constant function 1
in the case. The constant $x_c^\sigma$ can be determined by the normalization condition (\ref{norm}).
Letting $\sigma\rightarrow\infty$ and assuming $x_c^\sigma\to x_\gm$,
we get (\ref{harnlse367}) when $\sigma\to\infty$.
Plugging (\ref{harnlse367}) into the normalization condition (\ref{norm}),
we obtain (\ref{nlsehar689}).
\end{proof}

\begin{figure}[htbp]
\centerline{\psfig{figure=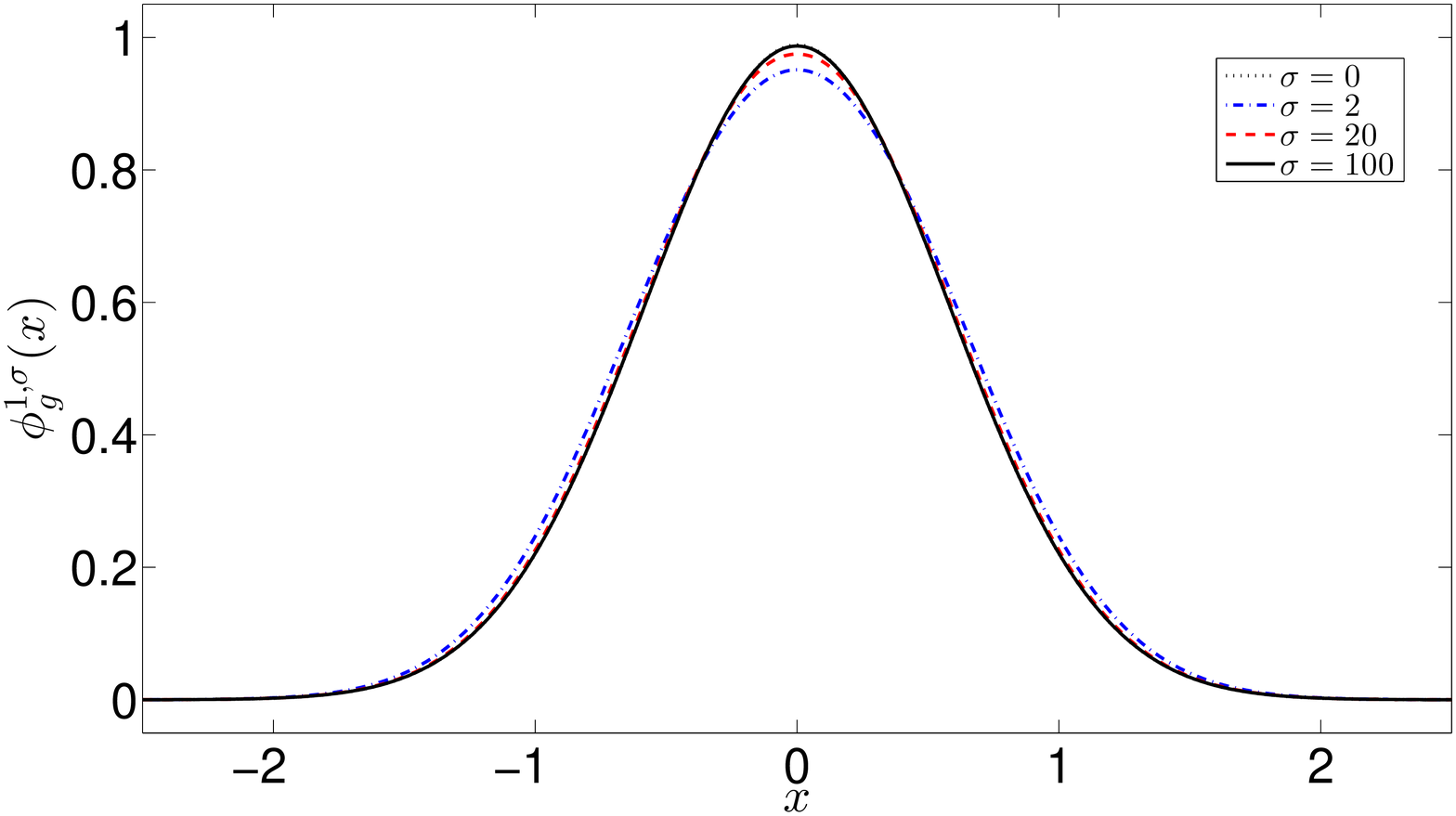,height=5cm,width=13cm,angle=0}}
\centerline{\psfig{figure=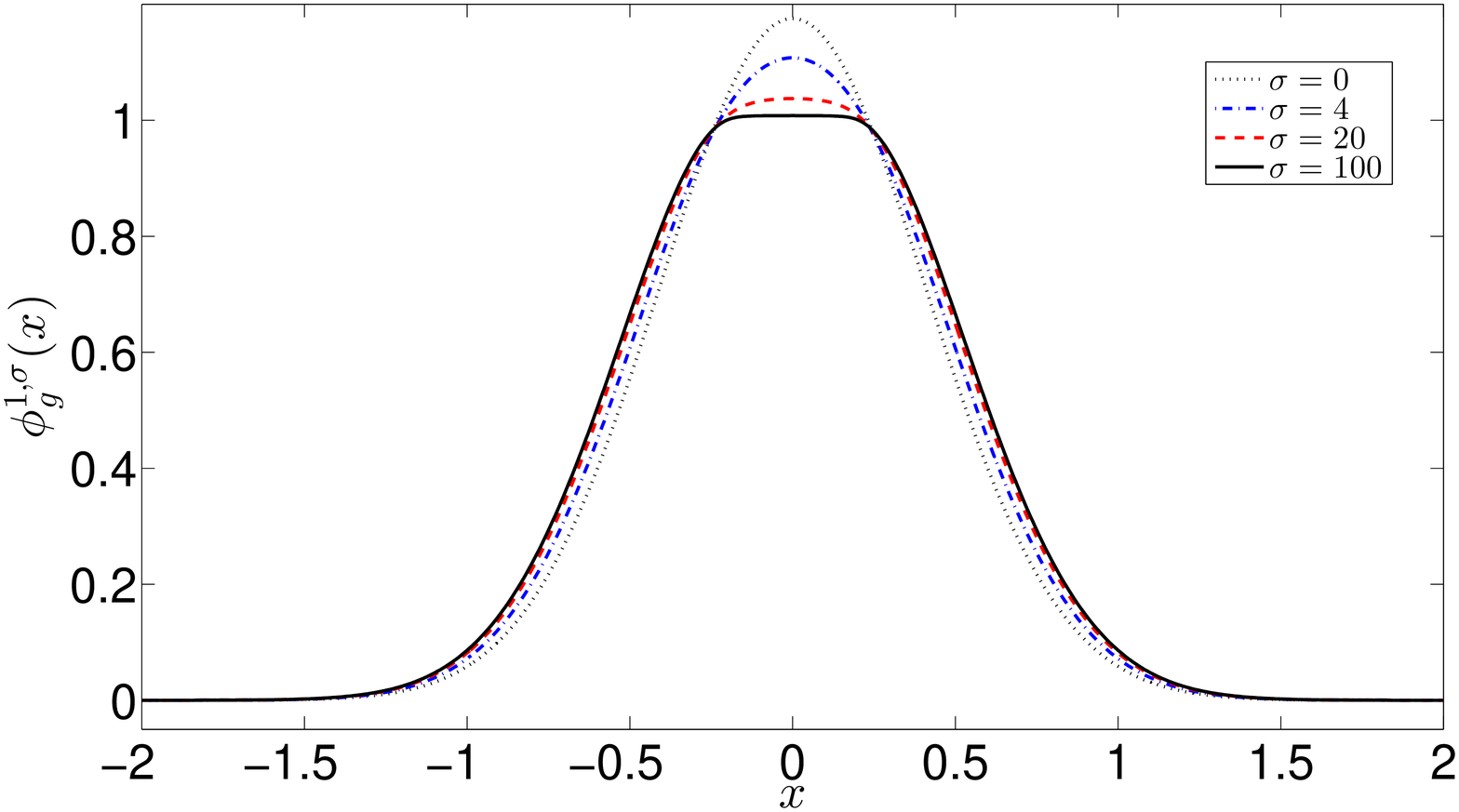,height=5cm,width=13cm,angle=0}}
\caption{Ground states of the NLSE in 1D with $\beta=1$ and $\gm=3<\pi$ (top) and
$\gm=6>\pi$ (bottom) for different nonlinearities, i.e. different values of $\sigma$.}
\label{fig:gGPE_har_gm3}
\end{figure}

\begin{figure}[htbp]
\centerline{\psfig{figure=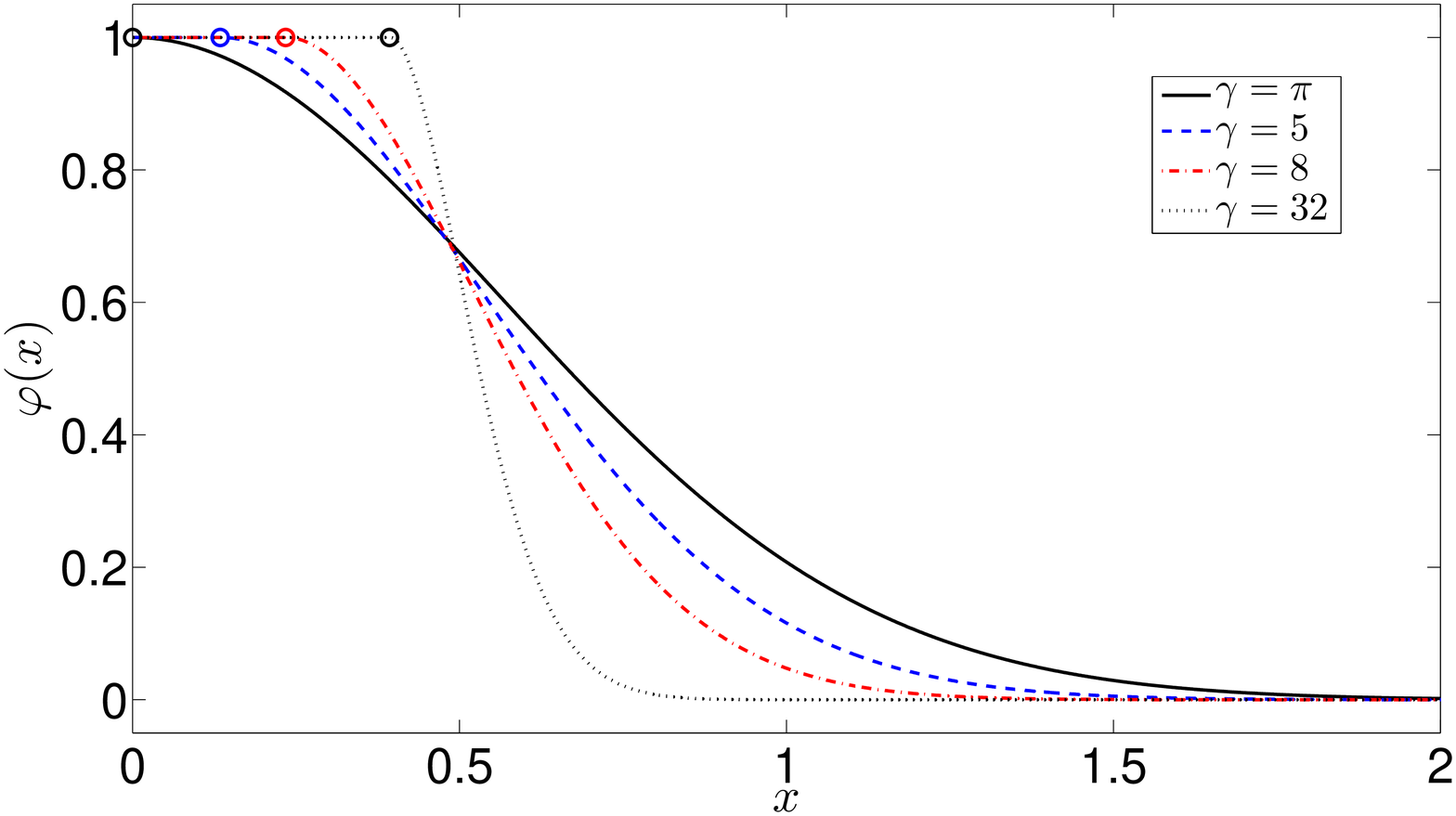,height=5cm,width=13cm,angle=0}}
\caption{Numerical solution of (\ref{sol:gGPE_har_eq}). The circles denote the points $(x_{\gm},1)$ corresponding to the different $\gm$'s.}
\label{fig:sol_varphi}
\end{figure}

\begin{figure}[htbp]
\centerline{\psfig{figure=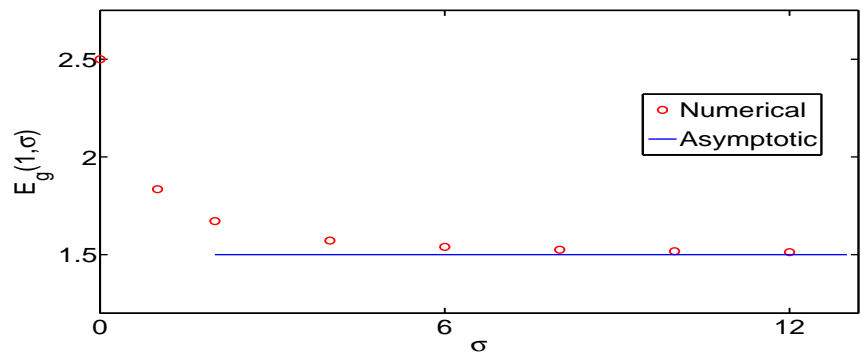,height=5cm,width=13cm,angle=0}}
\caption{The ground state energy of
the NLSE (\ref{eq:eig}) in 1D under a harmonic potential with $\beta=1$ and $\gm=3$ for different $\sigma$. }
\label{fig:ghar_E}
\end{figure}

\begin{figure}[!]
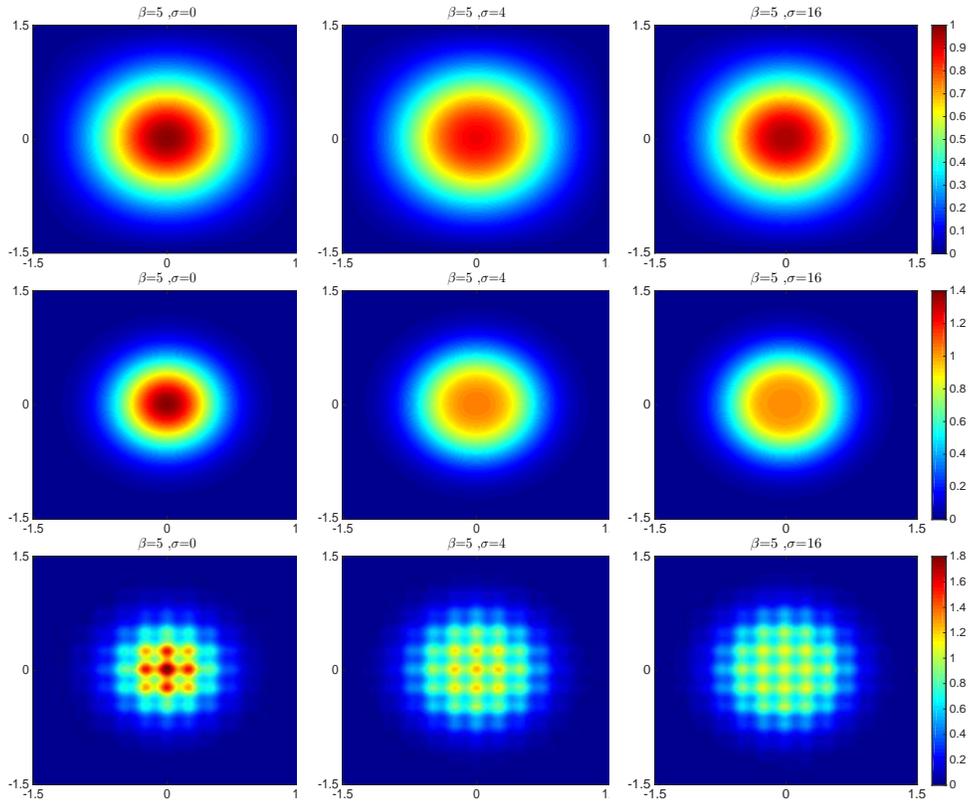

\centerline{\psfig{figure=FD2d-Bet-5-Sigma-0-har_3-gtol-1e-05.eps,height=3.5cm,width=4cm,angle=0}
\psfig{figure=FD2d-Bet-5-Sigma-4-har_3-gtol-1e-05.eps,height=3.5cm,width=4cm,angle=0}
\psfig{figure=FD2d-Bet-5-Sigma-16-har_3-gtol-1e-05.eps,height=3.5cm,width=4.5cm,angle=0}}
\centerline{\psfig{figure=FD2d-Bet-5-Sigma-0-har_6-gtol-1e-05.eps,height=3.5cm,width=4cm,angle=0}
\psfig{figure=FD2d-Bet-5-Sigma-4-har_6-gtol-1e-05.eps,height=3.5cm,width=4cm,angle=0}
\psfig{figure=FD2d-Bet-5-Sigma-16-har_6-gtol-1e-05.eps,height=3.5cm,width=4.5cm,angle=0}}
\centerline{\psfig{figure=FD2d-Bet-5-Sigma-0-lattice-gtol-1e-05.eps,height=3.5cm,width=4cm,angle=0}
\psfig{figure=FD2d-Bet-5-Sigma-4-lattice-gtol-1e-05.eps,height=3.5cm,width=4cm,angle=0}
\psfig{figure=FD2d-Bet-5-Sigma-16-lattice-gtol-1e-05.eps,height=3.5cm,width=4.5cm,angle=0}}
\caption{Ground states $\phi_g^{\beta,\sigma}$ under the harmonic potential $V(x,y)=9(x^2+y^2)/2$ (top row), $V(x,y)=18(x^2+y^2)$ (second row) and the lattice potential $V(x,y)=18(x^2+y^2)+100(\sin^2(4\pi x)+\sin^2(4\pi y))$ (third row) for $\beta=5$ and $\sigma=0$ (left column), $\sigma=4$ (middle column) and $\sigma=16$ (right column).}
\label{fig:har2d}
\end{figure}

In order to check our asymptotic results in Lemma \ref{gGPE_har_s},
we solve the time-independent NLSE (\ref{eq:eig}) numerically by using the
normalized gradient flow via backward Euler finite difference discretization
\cite{Bao2013,Bao_comp1,comp_gf,Wz1}
to find the ground states and the corresponding energy.
Figure~\ref{fig:gGPE_har_gm3} plots the ground states with $\beta=1$
for different $\sigma$ and $\gm$, Figure~\ref{fig:sol_varphi} shows the numerical solution of  (\ref{sol:gGPE_har_eq}) while
Figure~\ref{fig:ghar_E} depicts the energy asymptotics with $\beta=1$ and $\gm=3$ for different $\sigma$.
One thing that needs to be pointed out is that we can speculate the solution to (\ref{sol:gGPE_har_eq}) have the following properties from Figure~\ref{fig:sol_varphi}:\\
(i) $x_{\gm}\to0$, and $\psi_g^{\gm}(x)\to\psi_g^{\pi}(x)=e^{-\frac{\pi{x}^2}{2}}$ when $\gm\to\pi$,\\
(ii) $x_{\gm}\to0.5$ and $\psi_g^{\gm}(x)\to \psi_g^{\infty}(x)=1-1_{\{|x|\ge0.5\}}$ when $\gm\to\infty$.\\
Figure \ref{fig:har2d} plots the ground states in 2D under different potentials and with different nonlinearities. 
Again we observed different limiting patterns depending on the value of $\gm$, which is similar to the 1D case. 
We call this phenomenon to be the bifurcation of the ground state. 
The ground state $\phi_g^{\beta,\sigma}$ will converge to $\phi_g^{0}(\bx)$ \eqref{harg678} as $\sigma\to\infty$ if $\max|\phi_g^{0}(\bx)|<1$. Otherwise, it will converge to a function whose peaks are flat and the peak values are close to 1.

\section{Approximations under a box potential}\label{box:NLSE}
\setcounter{equation}{0}
\setcounter{figure}{0}
In this section, we take $\Omega=\prod_{j=1}^d(0,L_j)$ with $L_j>0$ for $j=1,\ldots,d$ and 
$V(\bx)\equiv0$ for $\bx\in\Omega$
in the NLSE (\ref{eq:eig}) with the homogeneous Dirichlet BC. 
For $\sigma=0$, the NLSE (\ref{eq:eig}) collapses to the linear Schr\"{o}dinger equation. 
From now on, we assume $\sigma>0$.

\subsection{For different $\beta$ under a fixed $\sigma>0$}\label{gGPE_box}
When $0\le \beta\ll1$, we have the following approximations for the ground state and the ground state energy.

\begin{lemma}\label{lemma:NLSE_box_weak}
When $0\le \beta\ll1$, i.e. weakly repulsive interaction regime,
the ground state $\phi_g^{\beta,\sigma}$ can be approximated as
\begin{equation}\label{boxg678}
\phi_g^{\beta,\sigma}(\mathbf{x})\approx\phi_g^{0}(\mathbf{x})
=2^{\frac{d}{2}}A_0\prod_{j=1}^d\sin\left(\frac{\pi{x_j}}{L_j}\right), \qquad \bx\in\bar{\Omega},
\end{equation}
where $A_0=\frac{1}{\sqrt{\prod_{j=1}^dL_j}}$ and the corresponding energy and chemical potential can be approximated as
\begin{align}\label{nlsebox1}
E_g(\beta,\sigma)&=\frac{\pi^2}{2}\sum_{j=1}^d\frac{1}{L_j^2}+
\frac{2^{d(\sigma+1)}A_0^{2\sigma}\beta}{(\sigma+1)\pi^d}
\left[\frac{\Gamma(\sigma+\frac{3}{2})\Gamma(\frac{1}{2})}{\Gamma(\sigma+2)}\right]^d+o(\beta),\\
\label{nlsebox2}
\mu_g(\beta,\sigma)&=\frac{\pi^2}{2}\sum_{j=1}^d\frac{1}{L_j^2}+
\frac{2^{d(\sigma+1)}A_0^{2\sigma}\beta}{\pi^d}
\left[\frac{\Gamma(\sigma+\frac{3}{2})\Gamma(\frac{1}{2})}{\Gamma(\sigma+2)}\right]^d+o(\beta).
\end{align}
\end{lemma}

\begin{proof}
When $\beta=0$, the NLSE \eqref{eq:eig} becomes linear and the ground state can be computed as $\phi_g^{0}(\bx)$.
When $0< \beta\ll1$, we can approximate the ground state $\phi_g^{\beta,\sigma}(\bx)$ by $\phi_g^{0}(\bx)$. 
Plugging (\ref{boxg678}) into (\ref{def:E}) and \eqref{def:mu} with $V(\bx)\equiv0$, we get 
(\ref{nlsebox1}) and \eqref{nlsebox2}, respectively. 
\end{proof}

\begin{remark}
Analogous to the harmonic potential case, we can show that $\phi_g^{\beta,\sigma}\to\phi_g^0$ in $H_1$ as $\beta\to0^+$ with general $\sigma>0$ or $\beta\to0$ with $\sigma>0$ satisfying $d\sigma<2$.
\end{remark}

When $\beta\gg1$, similar to the harmonic potential case, 
we adopt  the
TF approximation for the ground state. 

\begin{lemma}\label{lemma:NLSE_box_strong}
When $\beta\gg1$, i.e. strongly repulsive interaction regime,
the ground state can be approximated as
\begin{equation}\label{sol:gGPE_TF}
\phi_g^{\beta,\sigma}(\mathbf{x})\approx\phi_g^{TF}(\mathbf{x})=\frac{1}{\sqrt{\prod_{j=1}^dL_j}}, \qquad \bx\in\Omega,
\end{equation}
and the corresponding energy and chemical potential can be approximated as
\begin{equation}\label{gGPE_E_approx}
E_g(\beta,\sigma)\approx E_g^{TF}=\frac{A_0^{2\sigma}}{\sigma+1}\beta,\quad
\mu_g(\beta,\sigma)\approx\mu_g^{TF}=A_0^{2\sigma}\beta,\quad \beta\gg1.
\end{equation}
\end{lemma}

\begin{proof}
Similar to the proof in Lemma \ref{lemma:strong_nlse}, we drop the diffusion term in (\ref{eq:eig}) with $V(\bx)\equiv 0$ and get
\be
\mu_g^{TF}\phi_g^{\rm TF}(\mathbf{x})=\beta|\phi_g^{\rm TF}(\bx)|^{2\sigma}
\phi_g^{\rm TF}(\bx),\qquad \mathbf{x}\in\Omega.
\ee
Solving the above equation, we get
\be\label{proof:box_TF}
\phi_g^{TF}(\mathbf{x})=\left(\mu_g^{TF}/\beta\right)^{\frac{1}{2\sigma}},
\qquad \mathbf{x}\in\Omega.
\ee
$\mu_g^{TF}$ can be determined by plugging \eqref{proof:box_TF} into the normalization condition \eqref{norm}. 
And thus we obtain (\ref{sol:gGPE_TF}). Inserting
(\ref{sol:gGPE_TF}) into (\ref{def:mu}), we obtain $E_g^{TF}$.
\end{proof}

Note that the TF approximation (\ref{sol:gGPE_TF}) does not satisfy the homogeneous Dirichlet
BC. Therefore, the approximation is not uniformly accurate.
In fact, there exists a boundary layer along $\partial\Omega$
in the ground state when $\beta\gg1$.
Similar to the case of $\sigma=1$ \cite{BaoL,LZBao}, by using the matched asymptotic expansion method,
we can obtain an approximation which is uniformly accurate when $\beta\gg1$.

\begin{lemma}\label{lemma:mGPE_box_layer}
When $\beta\gg1$, i.e. strongly repulsive interaction regime,
a uniformly accurate ground state approximation can be given as
\be\label{sol:NLSE_box_MA}
\phi_g^{\beta,\sigma}(\bx)\approx\phi_g^{MA}(\mathbf{x})=\left(\frac{\mu_g^{MA}}{\beta}
\right)^{\frac{1}{2\sigma}}\prod_{j=1}^d\phi_{\sigma}(x_j; L_j,\mu_g^{MA}),
\ee
where $\phi_{\sigma}(x; L,\mu)=\varphi_{\sigma}\left(x\sqrt{\mu}\right)
+\varphi_{\sigma}\left((L-x)\sqrt{\mu}\right)-\varphi_{\sigma}\left(L\sqrt{\mu}\right)$, 
$\varphi_{\sigma}\left(L_j\sqrt{\mu_g^{MA}}\right)\approx 1$ and $\mu_g^{MA}\approx \mu_g(\beta,\sigma)=O(\beta)$ is the approximate chemical potential determined by the normalization condition \eqref{norm} and $\varphi_{\sigma}(x)$ satisfies the problem
\begin{equation}\label{gGPE_bl}
\begin{cases}
\varphi_{\sigma}(x)=-\frac{1}{2}\varphi_{\sigma}^{\prime\prime}(x)+\varphi_{\sigma}^{2\sigma+1}(x),
\qquad 0< x<+\infty,\\
\varphi_{\sigma}(0)=0,\qquad \lim\limits_{x\to+\infty}\varphi_{\sigma}(x)=1.
\end{cases}
\end{equation}
\end{lemma}

\begin{proof} For the simplicity of notation, we only prove it in 1D here.
Extension to higher dimensions can be done via dimension-by-dimension.
When $d=1$, there are two boundary layers in the ground state at $x_1=0$ and $x_1=L_1$,
respectively. Near $x_1=0$, we introduce the new variables
\be\label{changev}
\tilde{x}=x_1\sqrt{\mu_g(\beta,\sigma)},\quad
 \varphi_{\sigma}(\tilde{x})=\left(\frac{\beta}
 {\mu_g(\beta,\sigma)}\right)^{\frac{1}{2\sigma}}\phi(x_1),\quad x_1\ge0.
\ee
Substituting (\ref{changev}) into (\ref{eq:eig}) with $d=1$, $\Omega=(0,L_1)$ and
$V(\bx)\equiv 0$ and then removing all $\tilde{}$, we get (\ref{gGPE_bl}).
After obtaining the solution of (\ref{gGPE_bl}), an inner approximation
of the ground state near $x_1=0$ is given as
\be\label{blinn34}
\phi_g^{\beta,\sigma}(x_1)\approx
\left(\frac{\mu_g(\beta,\sigma)}{\beta}\right)^{\frac{1}{2\sigma}}
\varphi_{\sigma}\left(x_1\sqrt{\mu_g(\beta,\sigma)}\right),\qquad 0\le x_1\ll1.
\ee
Similarly, we can get the inner approximation
of the ground state near $x_1=L_1$ as
\be\label{blinn35}
\phi_g^{\beta,\sigma}(s)\approx
\left(\frac{\mu_g(\beta,\sigma)}{\beta}\right)^{\frac{1}{2\sigma}}
\varphi_{\sigma}\left(s\sqrt{\mu_g(\beta,\sigma)}\right),\quad 0\le s:=L_1-x_1\ll1.
\ee
Combining (\ref{blinn34}), (\ref{blinn35}) and the outer TF approximation (\ref{sol:gGPE_TF}),
using the matched asymptotic expansion method via denoting $\mu_g^{\rm TF}$ and $\mu_g(\beta,\sigma)$ by
$\mu_g^{\rm MA}$, we can obtain  (\ref{sol:NLSE_box_MA}).
\end{proof}

When $\sigma=1$, the solution of (\ref{gGPE_bl}) is given as
$\varphi_1(x)=\tanh(x)$ for $x\ge0$ \cite{BaoL,LZBao}.
For $0<\sigma\ne1$, in general, the problem (\ref{gGPE_bl})
cannot be solved explicitly. By a mathematical analysis (see details in Appendix \ref{mgpe_layer}),
we have

\begin{lemma}\label{varphi532} For any $\sigma>0$, the solution $\varphi_\sigma(x)$ of (\ref{gGPE_bl})
is a strictly increasing function for $x\ge0$ and satisfies
$\varphi_{\sigma}'(0)=\sqrt{\frac{2\sigma}{\sigma+1}}$. In addition,
when $\sigma\to+\infty$, we have
\begin{equation}\label{layer:gGPE}
\begin{aligned}
\varphi_\sigma(x)\to \varphi_{\infty}(x)=
\begin{cases}
\sin(\sqrt{2}x), &0\le x<\frac{\sqrt{2}\pi}{4},\\
1, &x\ge\frac{\sqrt{2}\pi}{4}.\\
\end{cases}
\end{aligned}
\end{equation}
\end{lemma}

Combining Lemmas~\ref{lemma:mGPE_box_layer} and \ref{varphi532},
we get the width of the boundary layers in the ground state in strongly repulsive
interaction regime, i.e. $\beta\gg1$,
is of order $O\left(\frac{1}{\sqrt{\beta}}\right)$ for any $\sigma>0$,
which is the same as in the GPE case \cite{BaoL,LZBao}.

As for the case $\beta\to-\infty$ with $d\sigma<2$, the limiting ground state should be the same as for the hamonic potential case since there is no external potential term in \eqref{eq:neg_infty}.   The details are omitted here for brevity.

Now we check the accuracy of the energy asymptotics in Lemma \ref{lemma:NLSE_box_weak} and \ref{lemma:NLSE_box_strong}. 
Figure~\ref{fig:gGPE_box_bE} shows the relative error of the energy approximation of the ground state,
i.e. $e(\beta):=\frac{|E_g(\beta,2)-E_g^{\rm app}|}{E_g(\beta,2)}$
when $\sigma=2$ for different $\beta\ge0$. As shown in the figure, the relative error goes to 0 as $\beta\to0$ or $\beta\to\infty$.

\begin{figure}[htbp]
\centerline{\psfig{figure=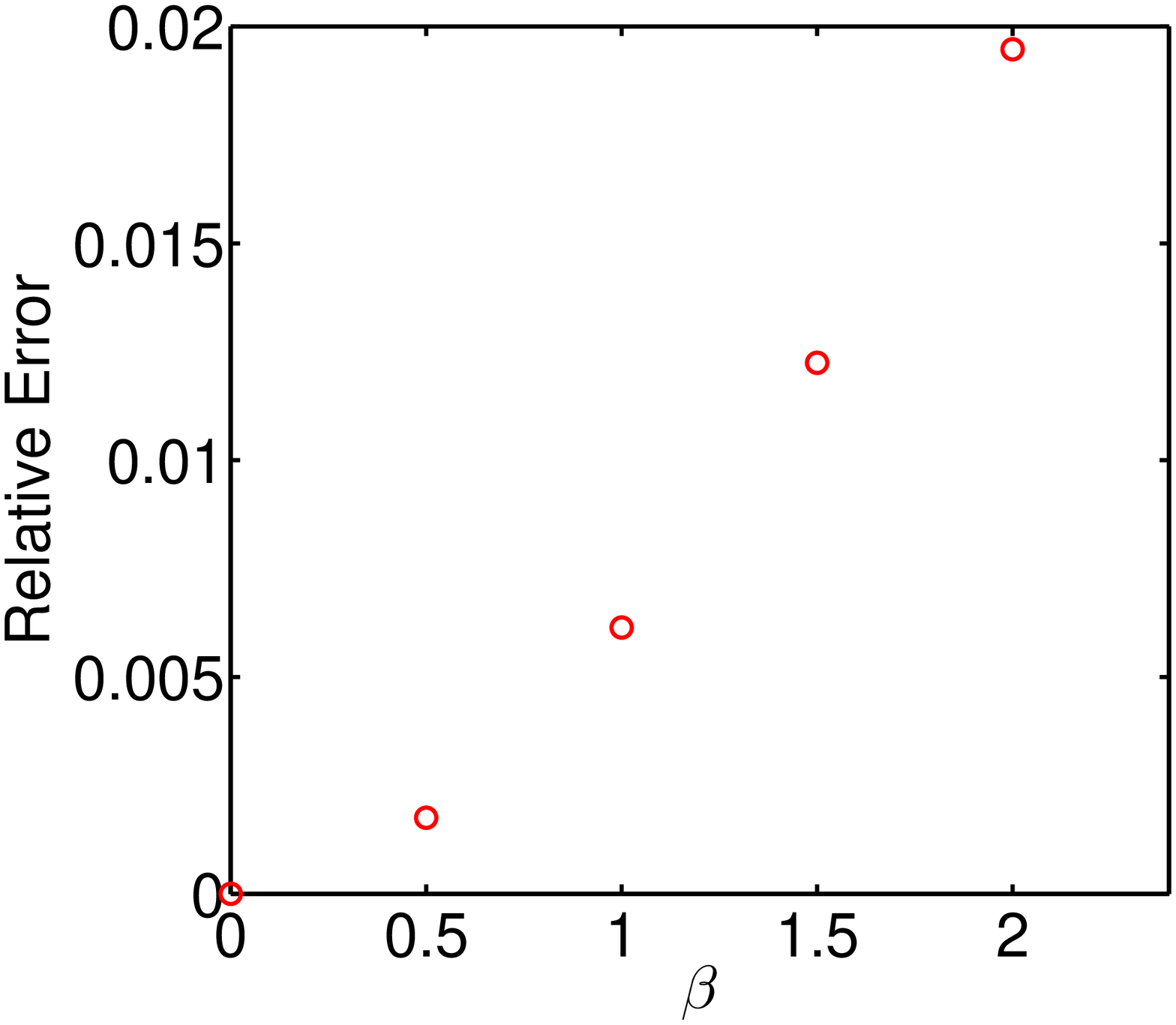,height=5cm,width=6.5cm,angle=0}
\psfig{figure=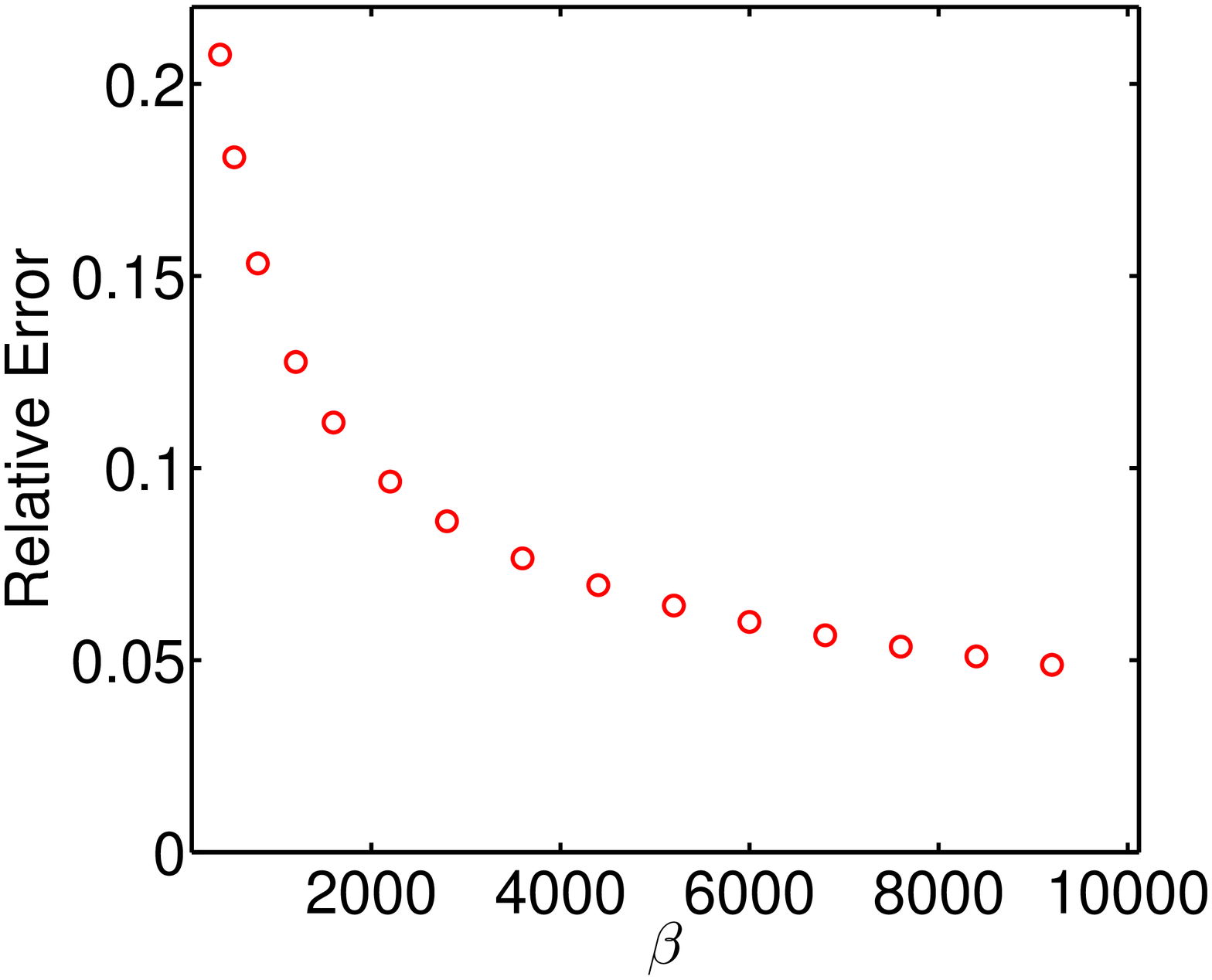,height=5cm,width=6.5cm,angle=0}}
\caption{Relative errors of the energy approximation of the ground state for
 the NLSE with $L=1$ and $\sigma=2$ in 1D with the
 box potential in the weak (left) and strong (right)
 interaction regimes.}
\label{fig:gGPE_box_bE}
\end{figure}

\subsection{When $\sigma\to\infty$ under a fixed $\beta>0$}
Here we assume $\beta>0$ is a given constant and we shall study the limit of the ground state
$\phi_g^{\beta,\sigma}$ when $\sigma\to\infty$.
For simplicity, we will only consider the NLSE in
1D on a bounded interval $\Omega=(0,L)$ with $L>0$.

\begin{lemma}\label{gGPE_box_s}
For any given $\beta>0$, when $\sigma\to\infty$, we have

(i) If $0<L<1$, the ground state converges to the TF approximation
\begin{align} \label{bnlse754}
&\phi_g^{\beta,\sigma}(x)\to \phi_g^{TF}(x)=\frac{1}{\sqrt{L}}, \qquad 0<x<L,\\
&\mu_g(\beta,\sigma)\approx\frac{\beta}
{L^{\sigma+1}} \to \infty,
\quad E_g(\beta,\sigma)\approx\frac{\beta}
{(\sigma+1)L^{\sigma+1}} \to \infty.
\end{align}

(ii) If $L\ge2$, the ground state converges to the linear approximation
\begin{align}\label{bnlse757}
&\phi_g^{\beta,\sigma}(x)\approx\phi_g^{0}(x)=\sqrt{\frac{2}{L}}
\sin\left(\frac{\pi{x}}{L}\right), \qquad 0\le x\le L,\\
&\mu_g(\beta,\sigma)\approx\frac{\pi^2}{2L^2}+\frac{2\beta}{\pi}
\left(\frac{2}{L}\right)^{\sigma}\left[\frac{\Gamma(\sigma+\frac{3}{2})
\Gamma(\frac{1}{2})}{\Gamma(\sigma+2)}\right]\to\frac{\pi^2}{2L^2}, \\
&{E}_g(\beta,\sigma)\approx\frac{\pi^2}{2L^2}+\frac{2\beta}{(\sigma+1)\pi}
\left(\frac{2}{L}\right)^{\sigma}\left[\frac{\Gamma(\sigma+\frac{3}{2})
\Gamma(\frac{1}{2})}{\Gamma(\sigma+2)}\right]\to\frac{\pi^2}{2L^2}.
\end{align}

(iii) If $1<L<2$, the ground state converges to
\begin{align}\label{nlsebox679}
&\phi_g^{\beta,\sigma}(x)\to \phi_g^{\infty}(x)=
\begin{cases}
\sin(\frac{\pi{x}}{2(L-1)}), &0\le{x}<L-1,\\
1, &L-1\le{x}\le1,\\
\sin(\frac{\pi(L+x-2)}{2(L-1)}), &1<x\le{L},
\end{cases}\\
\label{nlsebox689}
&\mu_g(\beta,\sigma)\to
\frac{\pi^2}{8(L-1)^2},\quad
{E}_g(\beta,\sigma)\to\frac{\pi^2}{8(L-1)}.
\end{align}
\end{lemma}

\begin{proof}
Similar to the proof in Lemma \ref{gGPE_har_s}, 
we need to determine which term on the left hand side of (\ref{eq:eig})
is negligible when $\sigma>>1$. In the region where $|\phi(x)|<1$, the nonlinear term can
be dropped and we get the linear approximation, whose solution is the sine function.
In the region where $|\phi(x)|>1$, the diffusion term can be dropped and we
get the TF approximation, whose solution is a constant. Therefore, there are three possible cases concerning the
limit $\phi_g^{\beta,\sigma}(x)\to\phi^{\rm{app}}(x)$ for $0<x<L$ when $\sigma\to+\infty$:
(i) $|\phi^{\rm{app}}(x)|\ge1$ for all $x\in(0,L)$, (ii) $|\phi^{\rm app}(x)|\le1$ for all $x\in(0,L)$,
and (iii) there exists $0<x_c<L/2$ such that $|\phi^{\rm{app}}(x)|\ge1$ for $x\in[x_c,L-x_c]$
and $|\phi^{\rm{app}}(x)|<1$ otherwise.

(i) When $0<L\le1$, the TF approximation suggests that  $\phi^{\rm app}(x)=\sqrt{1/L}\ge1$ for
$0<x<L$. Note that the requirement that $\inf\limits_{0<x<L}\phi^{\rm app}(x)\ge1$ implies that $L\le1$.
Therefore, we get the necessary and sufficient condition about $L$ for
(\ref{bnlse754}) to be true.

(ii) When $L\ge2$, the linear approximation suggests that
$\phi^{\rm app}(x)=\frac{2}{L}\sin\left(\frac{\pi x}{L}\right)\le1$ for
$0<x<L$. Note that the requirement that $\sup\limits_{0<x<L}\phi^{\rm app}(x)\le1$ implies that $L\ge2$.
Therefore, we get the necessary and sufficient condition about $L$ for
(\ref{bnlse757}) to be true.

(iii) When $1<L<2$, we may expect neither the linear approximation nor the
TF approximation is valid for $0<x<L$.
Instead, a combination of the linear approximation and TF approximation should be used.
To be more specific, for any fixed $\sigma>0$, when
$\beta>>1$, there exists a constant $x_c^\sigma$ such that when $x\in(0,x_c^\sigma)$
or $x\in[L-x_c^\sigma,L]$,
the linear approximation is used; and  when $x\in[x_c^\sigma,L-x_c^\sigma]$, the TF approximation
which is a constant, should be used.
For $x\in[x_c^\sigma,L-x_c^\sigma]$, assuming that $\phi^\sigma_g(x)=A_\sigma$ with $A_\sigma>0$ is a constant to be determined,
the approximate solution in $(0,x_c^\sigma)$ must be $\phi^{\sigma}_g(x)=
A_\sigma\sin\left(\frac{\pi{x}}{2x_c^\sigma}\right)$ in order to make the combined solution to be $C^1$ continuous.
Now we need to determine the value of $A_\sigma$ and $x_c^\sigma$.  By the normalization condition (\ref{norm}), we get
\bea
\frac{1}{2}=\int_0^{\frac{L}{2}}|\phi^{\sigma}_g(x)|^2dx=
\int_0^{x_c^\sigma}|\phi^{\sigma}_g(x)|^2dx+\int_{x_c^\sigma}^{\frac{L}{2}}|\phi^{\sigma}_g(x)|^2dx
=A_\sigma^2\left(\frac{L}{2}-\frac{x_c^\sigma}{2}\right).\quad
\eea
Thus, we have
\be\label{gGPE_box_p1}
A_\sigma=\frac{1}{\sqrt{L-x_c^\sigma}}.
\ee
In $[0,x_c^\sigma)$, dropping the nonlinear term in (\ref{eq:eig})
and substituting the approximate solution into it, we get
\begin{equation}\label{mu_gGPE_box_in_proof}
\mu_g=\frac{\pi^2}{8(x_c^\sigma)^2}.
\end{equation}
In $[x_c^\sigma,L-x_c^\sigma]$, dropping the  diffusion term in (\ref{eq:eig}), we get
\begin{equation}\label{mg3456}
\mu_g=\beta A_\sigma^{2\sigma}.
\end{equation}
Combining (\ref{mu_gGPE_box_in_proof}) and (\ref{mg3456}), we obtain
\begin{equation}\label{gGPE_box_p2}
A_\sigma^2=\left(\frac{\pi^2}{8\beta (x_c^\sigma)^2}\right)^{1/\sigma}.
\end{equation}
Inserting (\ref{gGPE_box_p1}) into (\ref{gGPE_box_p2}), we have
\be
\left(\frac{\pi^2}{8\beta(x_c^\sigma)^2}\right)^{\frac{1}{\sigma}}=\frac{1}{L-x_c^\sigma}.
\ee
Letting $\sigma\rightarrow\infty$ and assuming $x_c^\sigma\to x_c$ and $A_\sigma\to A$,
we have $1=\frac{1}{L-x_c}$, which implies that $x_c=L-1$ and we get $A=1$ via (\ref{gGPE_box_p1})
when $\sigma\to\infty$. Thus we get (\ref{nlsebox679}) when $\sigma\to\infty$.
$\mu_g(\beta,\infty)$ can be computed from (\ref{mu_gGPE_box_in_proof}) and $E_g(\beta,\infty)$ is from definition (\ref{def:E}), i.e. 
$$E_g(\beta,\infty)=\lim_{\sigma\to\infty}\int_0^L\left[\frac{1}{2}|\nabla
\phi_g^{\beta,\sigma}|^2+\frac{\beta}{\sigma+1}|\phi_g^{\beta,\sigma}|^{2\sigma+2}\right]dx.$$
However, direct computation by using (\ref{nlsebox679}) may be unreasonable because we cannot get the limit of $\int_{L-1}^1|\phi_g^{\beta,\sigma}|^{2\sigma+2}dx$. In fact, to get $E_g(\beta,\infty)$, we only need the upper limit of $\int_{L-1}^1|\phi_g^{\beta,\sigma}|^{2\sigma+2}dx$ is bounded,
which is true because
$$0\le\limsup_{\sigma\to\infty}\beta\int_{L-1}^1|\phi_g^{\beta,\sigma}|^{2\sigma+2}dx
\le\lim_{\sigma\to\infty}\mu_g(\beta,\sigma)=\frac{\pi^2}{8(L-1)^2}.$$
It follows that $\lim_{\sigma\to\infty}\int_0^L\frac{\beta}{\sigma+1}|\phi_g^{\beta,\sigma}|^{2\sigma+2}dx=0$ and
$$E_g(\beta,\infty)
=\lim_{\sigma\to\infty}\int_0^L\frac{1}{2}|\nabla\phi_g^{\beta,\sigma}|^2dx
\approx\int_0^L\frac{1}{2}|\nabla\phi_g^{\beta,\infty}|^2dx
=\frac{\pi^2}{8(L-1)}.$$
\end{proof}

\begin{figure}[htbp]
\centerline{\psfig{figure=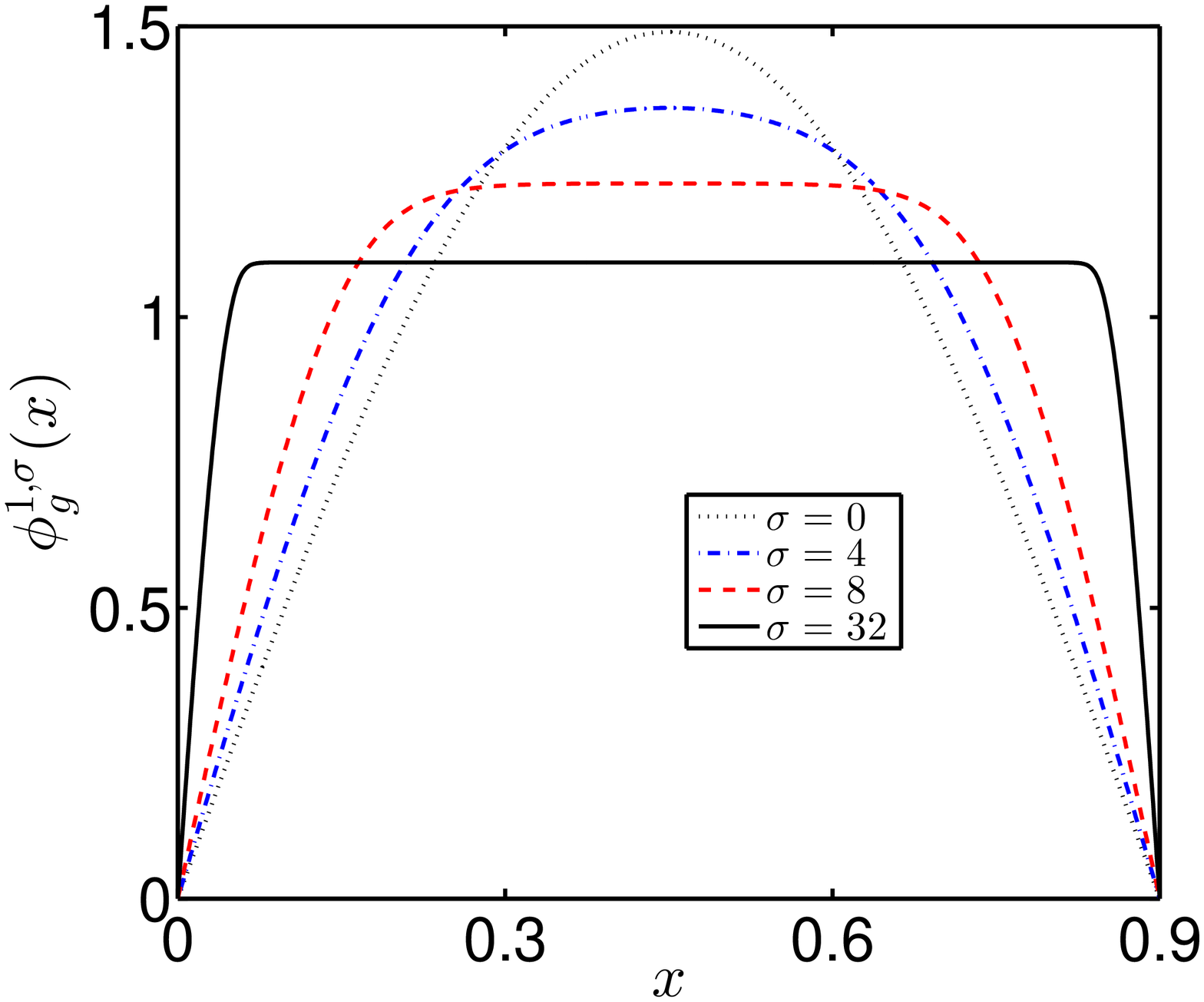,height=5cm,width=6.5cm,angle=0}
\psfig{figure=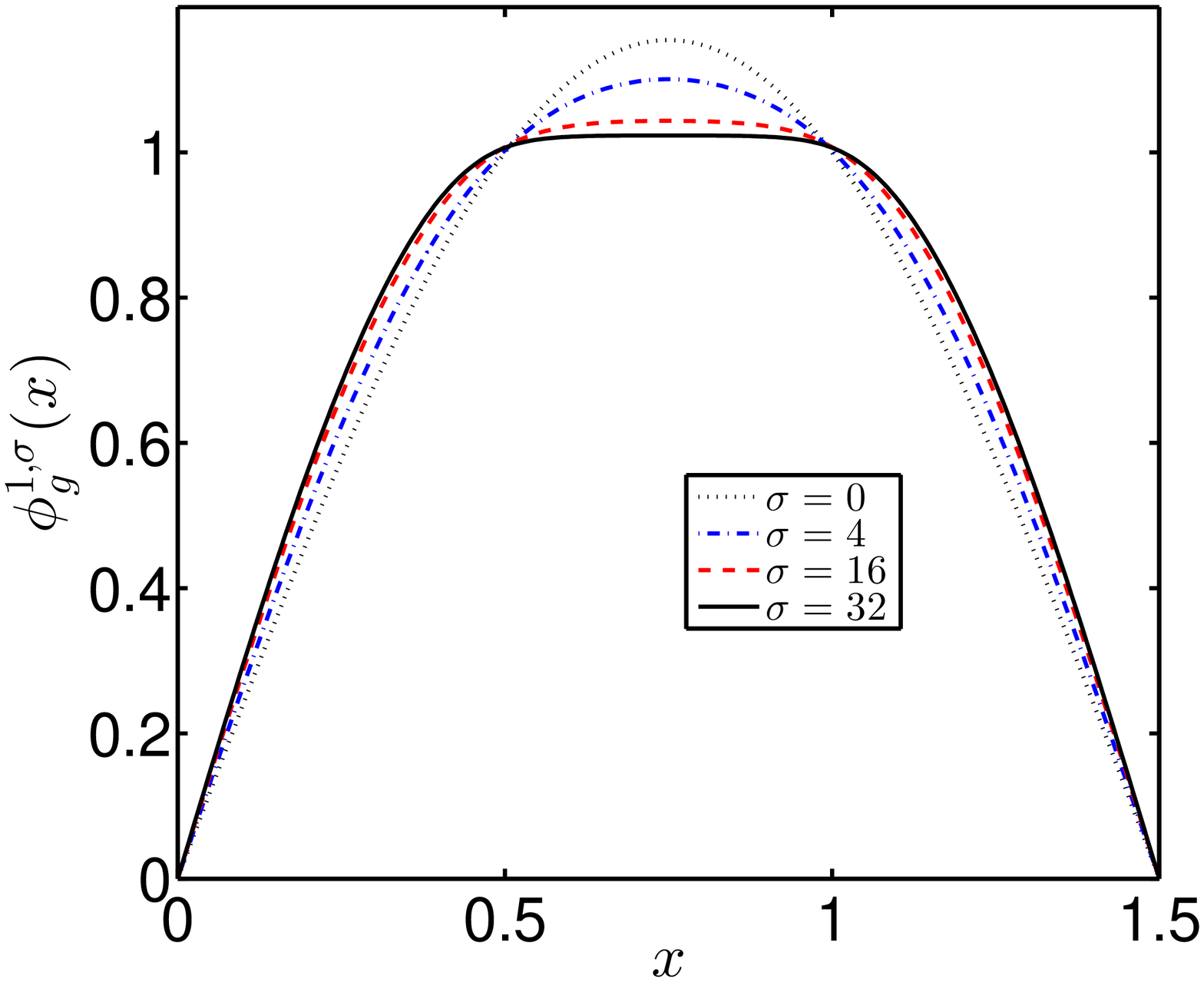,height=5cm,width=6.5cm,angle=0}}
\centerline{\psfig{figure=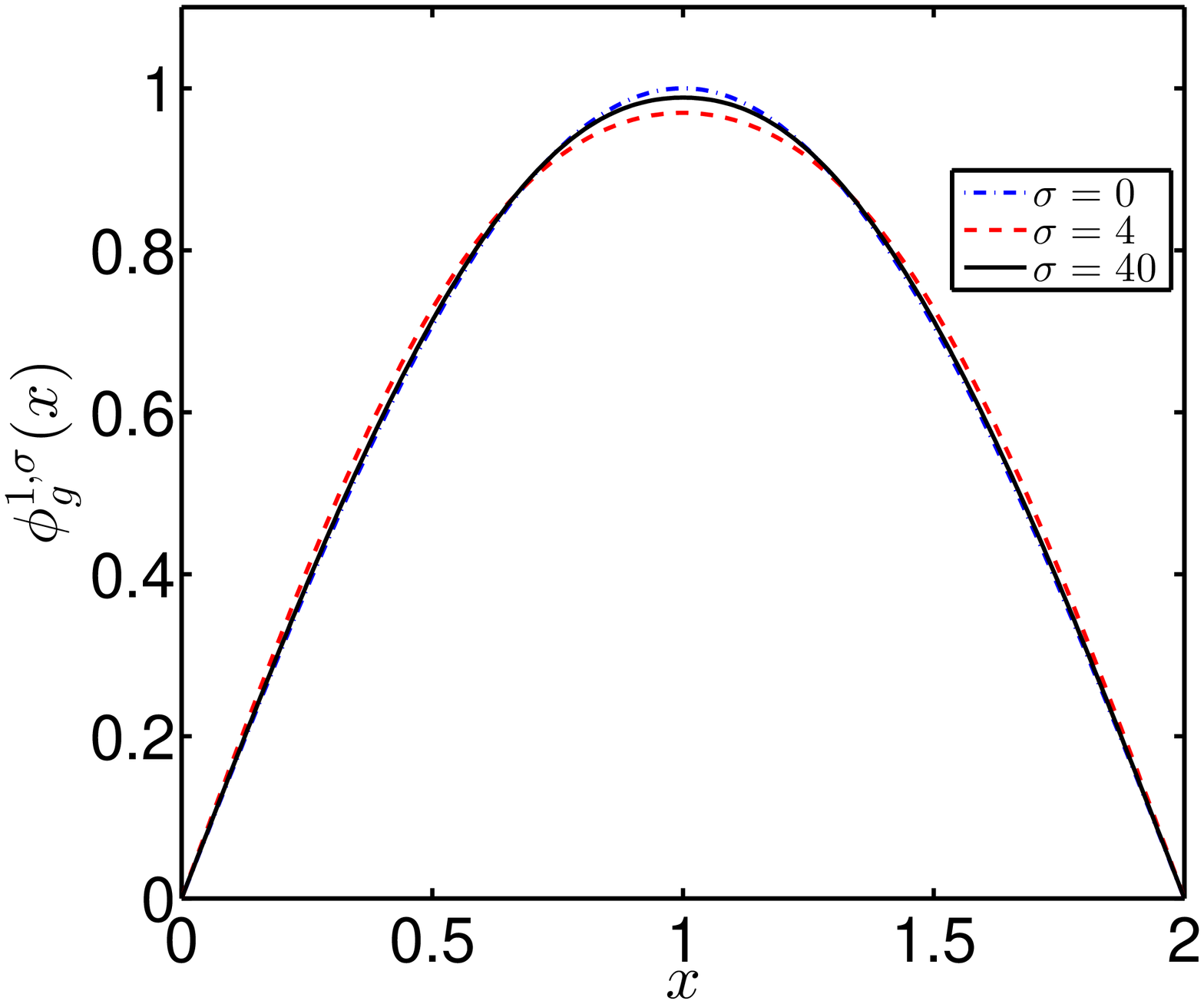,height=5cm,width=6.5cm,angle=0}}
\caption{Ground states of the NLSE  in 1D with $\beta=1$ and the box potential
for different $\sigma$ and $L=0.9<1$ (upper left), $1<L=1.5<2$ (upper right)
and $L=2.0$ (bottom).}
\label{fig:gGPE_box1}
\end{figure}

Now we check our asymptotic results in Lemma \ref{gGPE_box_s}.
Figure~\ref{fig:gGPE_box1} plots the ground states with $\beta=1$
for different $\sigma$ and $L$, and
Figure~\ref{fig:gGPE_box_sE} depicts the ground state energy with $\beta=1$ and $L=1.2$ for
different $\sigma$.
From Figures \ref{fig:gGPE_box1} and \ref{fig:gGPE_box_sE}, our asymptotic results in Lemma \ref{gGPE_box_s} are confirmed.
Figure \ref{fig:box2d} plots the ground state computed in 2D. Similar to the 1D case, the bifurcation of the ground state is observed.

\begin{figure}[htbp]
\centerline{\psfig{figure=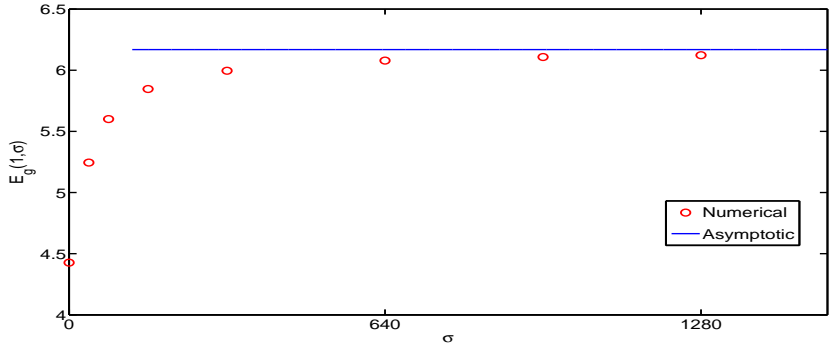,height=5cm,width=13cm,angle=0}}
\caption{Ground state energy of the NLSE  in 1D with $\beta=1$, $L=1.2$ and different $\sigma$ under the box potential.}
\label{fig:gGPE_box_sE}
\end{figure}

\begin{figure}[!]
\centerline{\psfig{figure=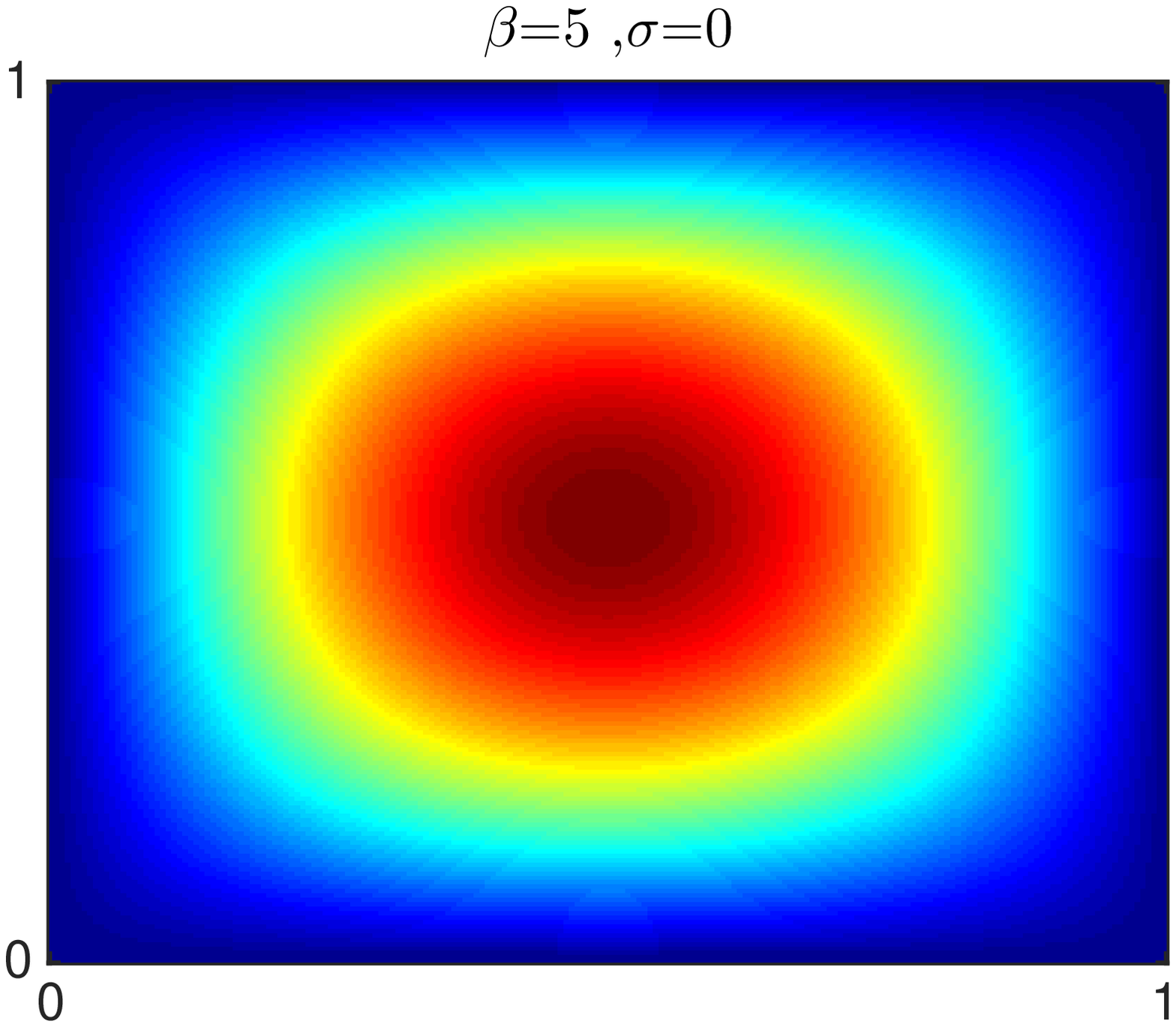,height=3.5cm,width=4cm,angle=0}
\psfig{figure=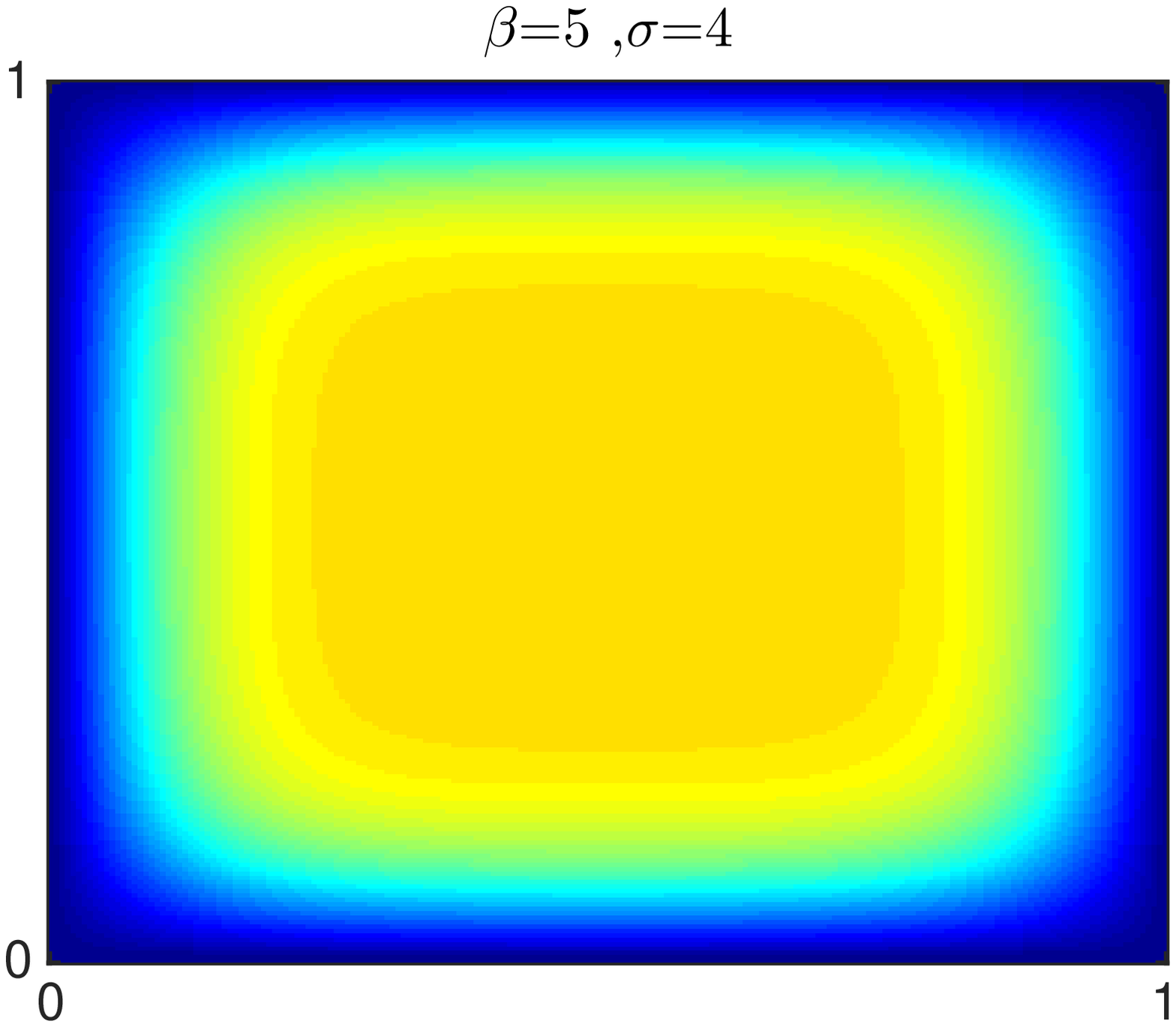,height=3.5cm,width=4cm,angle=0}
\psfig{figure=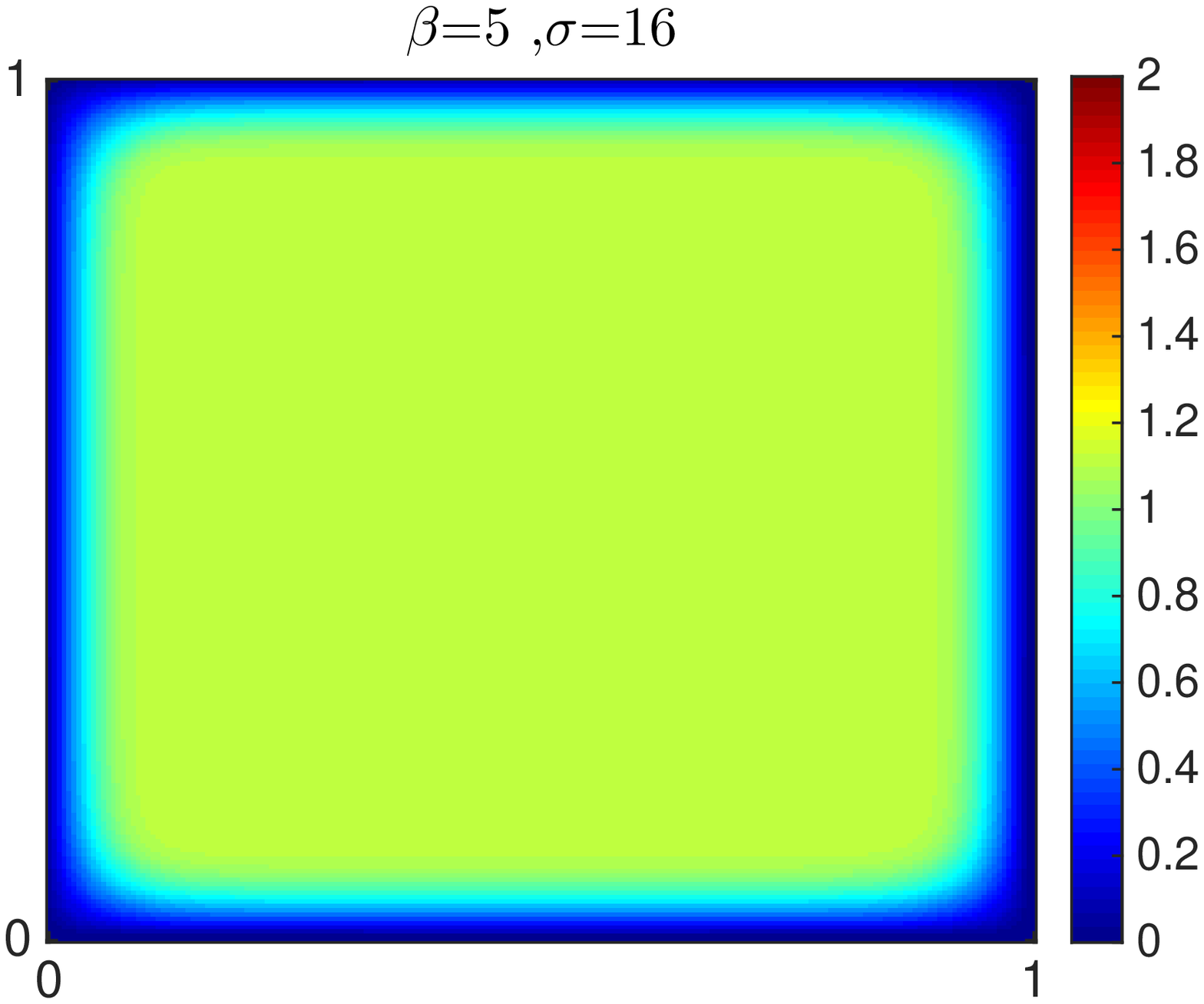,height=3.5cm,width=4.5cm,angle=0}}
\centerline{\psfig{figure=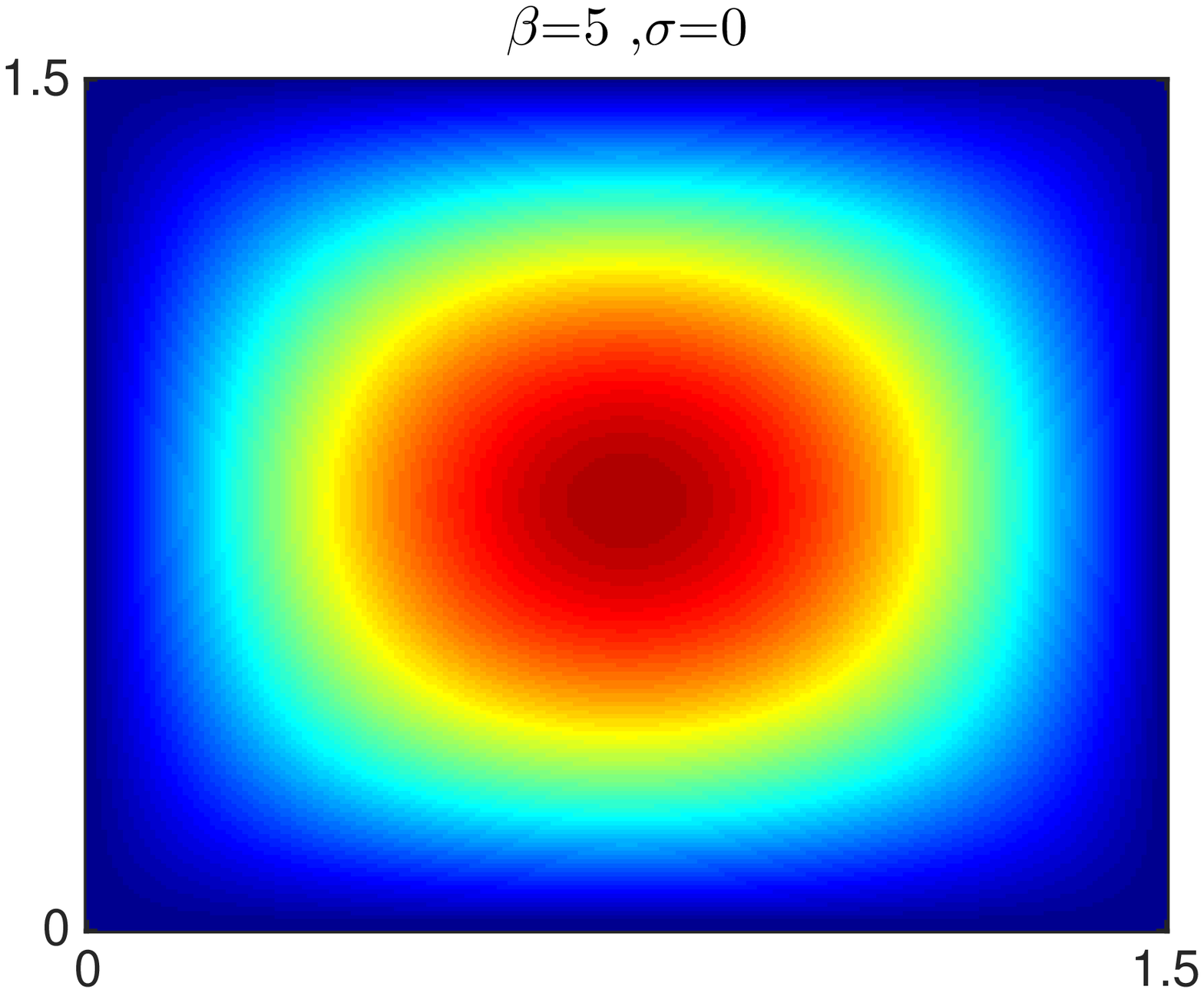,height=3.5cm,width=4cm,angle=0}
\psfig{figure=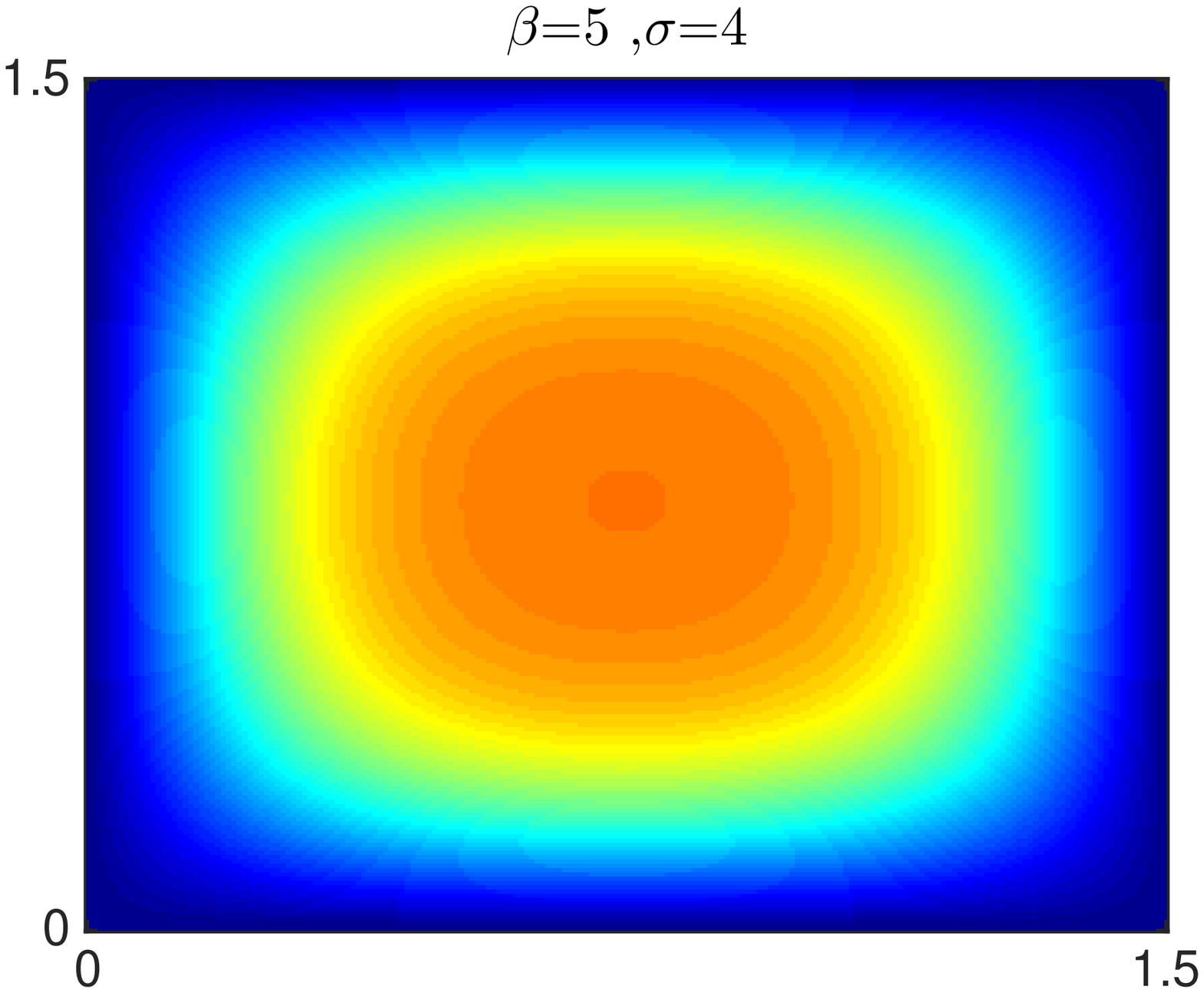,height=3.5cm,width=4cm,angle=0}
\psfig{figure=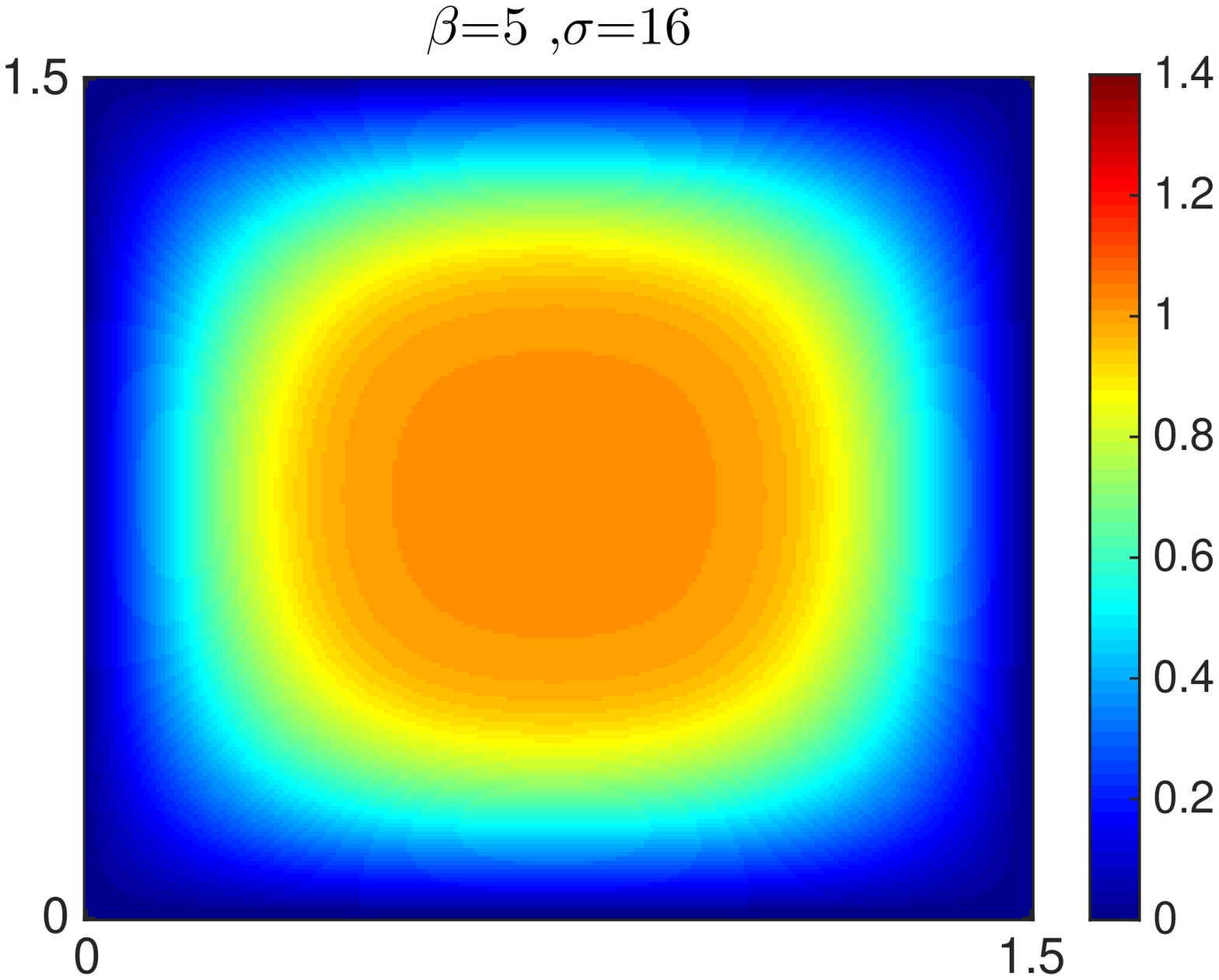,height=3.5cm,width=4.5cm,angle=0}}
\centerline{\psfig{figure=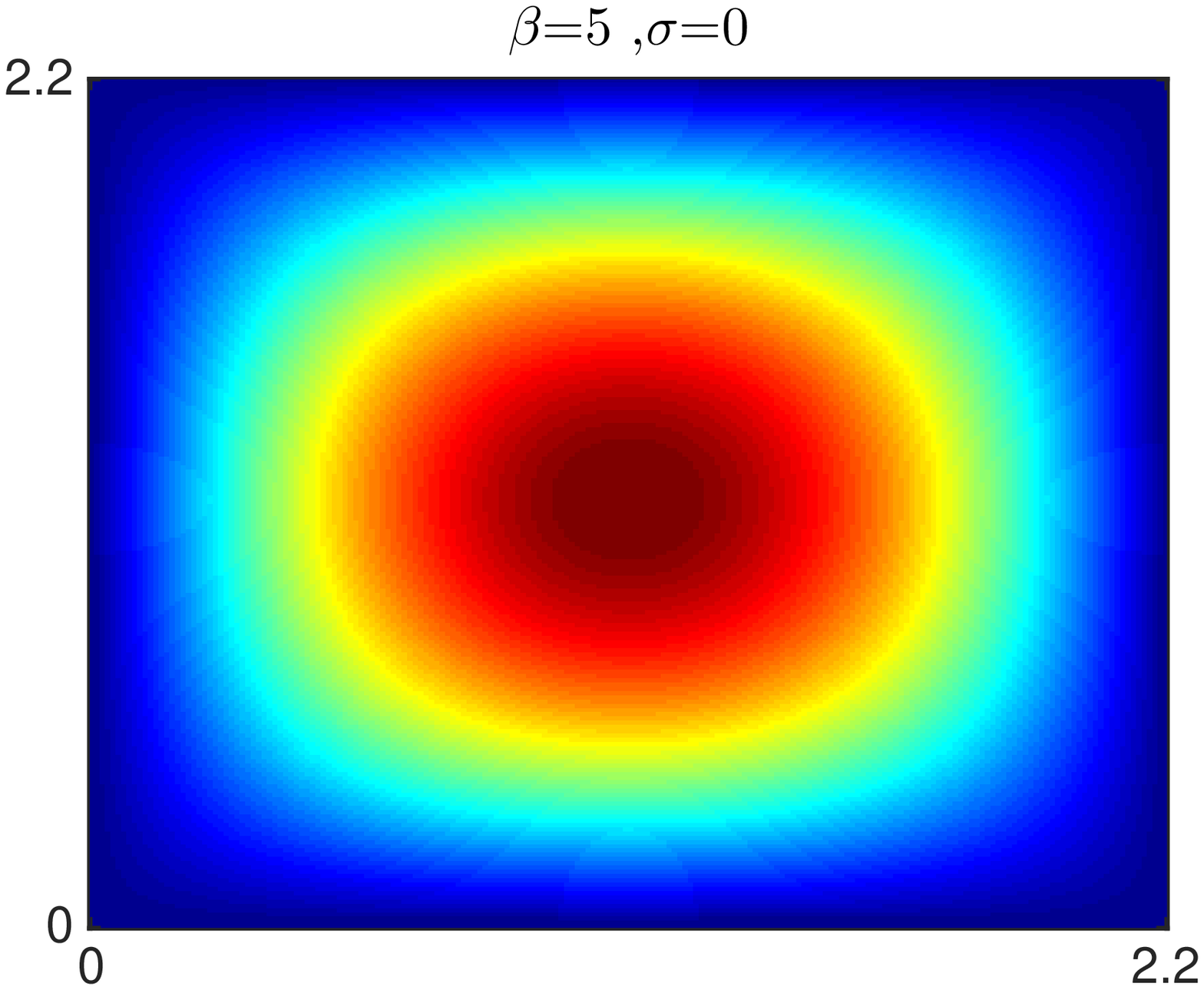,height=3.5cm,width=4cm,angle=0}
\psfig{figure=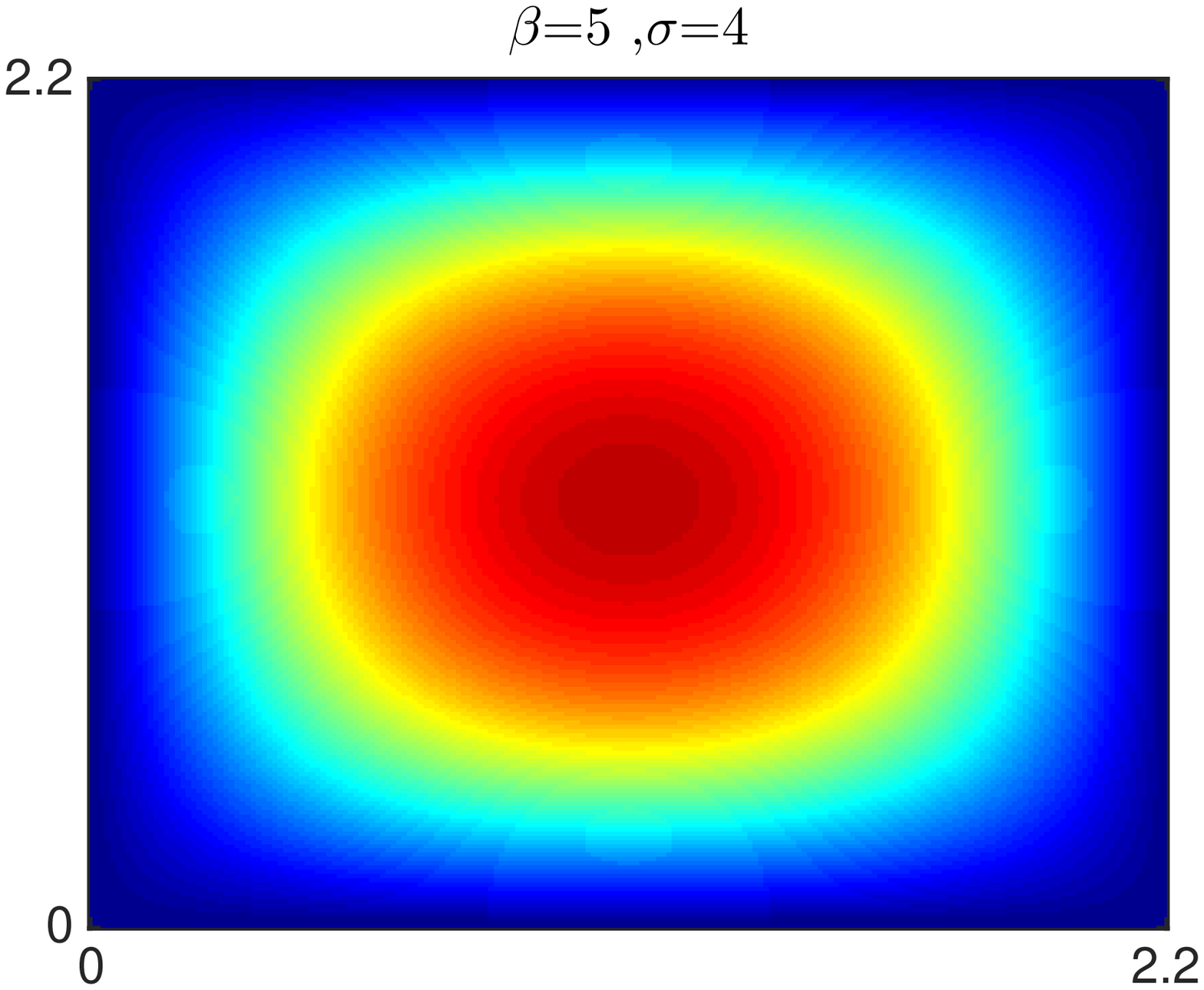,height=3.5cm,width=4cm,angle=0}
\psfig{figure=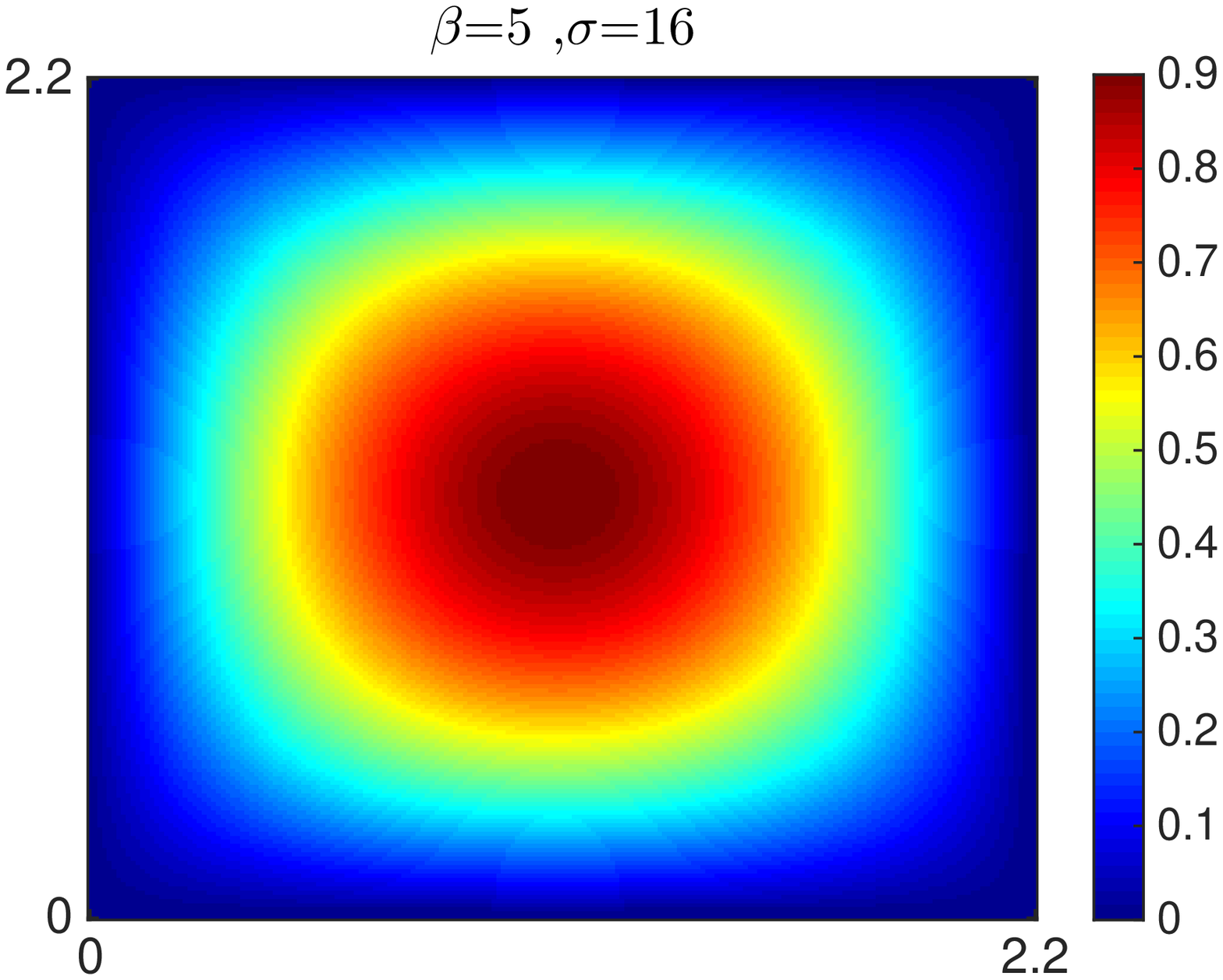,height=3.5cm,width=4.5cm,angle=0}}
\caption{Ground states $\phi_g^{\beta,\sigma}$ under the box potential in $\Omega=(0,1)^2$ (top row), $\Omega=(0,1.5)^2$ (second row) and $\Omega=(0,2.2)^2$ (third row) for $\beta=5$ and $\sigma=0$ (left column), $\sigma=4$ (middle column) and $\sigma=16$ (right column).}
\label{fig:box2d}
\end{figure}

\section{Conclusions}\label{conclusion}
We generalized the existence and uniqueness of the ground state from the Gross-Pitaevskii equation (GPE) to the nonlinear Schr\"{o}dinger equation (NLSE).
 In addition, we studied asymptotically
the ground states and their corresponding energy and chemical potential of
the NLSE with different nonlinearities.
For NLSE with a fixed nonlinearity under a box or a harmonic potential, we derived explicitly the approximations of the ground state and the corresponding energy and chemical potential. 
If we let the nonlinearity component $\sigma\to\infty$ and fix the interaction strength, 
we observed different limiting patterns and called this phenomenon the ``bifurcation of the ground state". The characterization of the ground state in 1D in each pattern and the corresponding leading order energy asymptotics were derived explicitly and verified numerically.
Similar phenomenon was observed in higher dimension case as well.


\appendix

\section{Proof of Lemma \ref{varphi532}} \label{mgpe_layer}
\setcounter{equation}{0}
\setcounter{figure}{0}
Multiplying (\ref{gGPE_bl}) by $\varphi_{\sigma}^\prime(x)$, we get
\be
\frac{1}{2}\left(\varphi_{\sigma}^2(x)\right)^\prime=
-\frac{1}{4}\left(\left(\varphi_{\sigma}^\prime(x)\right)^2\right)^\prime+
\frac{1}{2\sigma+2}\left(\varphi_{\sigma}^{2\sigma+2}(x)\right)^\prime, \qquad x>0.
\ee
Therefore, we have
\be\label{solvarp78}
\varphi_{\sigma}^2(x)=-\frac{1}{2}\left(\varphi_{\sigma}^\prime(x)\right)^2
+\frac{1}{\sigma+1}\varphi_{\sigma}^{2\sigma+2}(x)+C, \qquad x\ge0,
\ee
where $C$ is the integrating constant.
When $x\rightarrow+\infty$, we have $\varphi_{\sigma}(x)\rightarrow1$ and $\varphi_{\sigma}^\prime(x)
\rightarrow0$. So we get $C=\frac{\sigma}{1+\sigma}$. Letting $x=0$ in (\ref{solvarp78}), we get
\be\label{vph0123}
\varphi_{\sigma}^\prime(0)=\sqrt{\frac{2\sigma}{\sigma+1}}, \qquad \sigma>0.
\ee
For $\sigma>0$, by using the maximum principle,
we have $0\le \varphi_\sigma(x)<1$ for $x\ge0$.
When $\sigma\rightarrow\infty$, we have
$\varphi_{\sigma}^{2\sigma+1}(x)\rightarrow0$ for $x\ge0$.
Therefore, when $\sigma\rightarrow\infty$, noting (\ref{vph0123}), the problem
(\ref{gGPE_bl}) converges to the following linear problem:
\begin{equation}
\begin{cases}
\varphi_{\infty}(x)=-\frac{1}{2}\varphi_{\infty}^{\prime\prime}(x), \qquad x>0,\\
\varphi_{\infty}(0)=0,\qquad \varphi_{\infty}^\prime(0)=\sqrt{2}.
\end{cases}
\end{equation}
Solving this problem, we obtain (\ref{layer:gGPE}) immediately. \hfill $\Box$

 To illustrate the solution $\varphi_\sigma(x)$ of (\ref{gGPE_bl}),
Figure \ref{figvarphis} plots $\varphi_\sigma(x)$ obtained numerically
for different $\sigma$. From this figure, we can see that:
(i) For any $\sigma>0$, $\varphi_\sigma(x)$ is a monotonically increasing function.
(ii) When $\sigma\to+\infty$, $\varphi_\sigma(x)$ converges to $\varphi_\infty(x)$
uniformly for $x\ge0$ (cf. Figure \ref{figvarphis}).

\begin{figure}[htbp]
\centerline{\psfig{figure=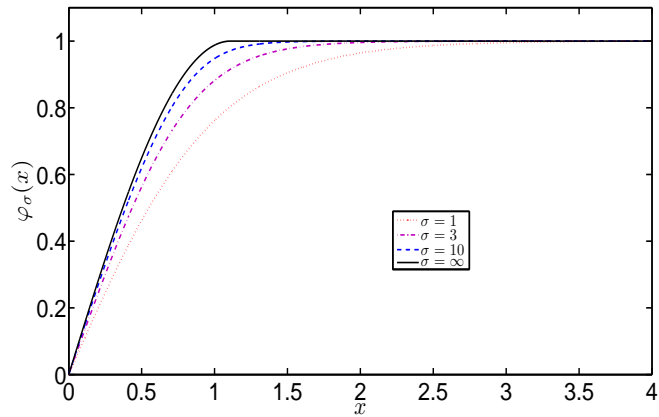,height=6cm,width=10cm,angle=0}}
\caption{Plots of the solution $\varphi_\sigma(x)$ of the problem
(\ref{gGPE_bl}) for $\sigma=1, 3, 10, \infty$
(with the order from right to left).}
\label{figvarphis}
\end{figure}

\section*{Acknowledgments}
This work was supported by the Academic Research Fund of Ministry of Education of Singapore grant No. R-146-000-223-112 and 
I would like to specially thank Prof.  Weizhu Bao in National University of Singapore, who significant contributed to the paper with his valuable comments.

\end{document}